\documentclass[11pt,reqno]{article}

\usepackage[leqno]{amsmath}
\usepackage{amssymb,xfrac}
\usepackage{amsthm}
\usepackage{ifthen} 
\usepackage[pagebackref,breaklinks,unicode]{hyperref} 
    \renewcommand*{\backrefalt}[4]{\ifcase #1 (Not cited).\or (Cited p.~#2).\else (Cited pp.~#2).\fi} 
\usepackage[capitalise,noabbrev]{cleveref}
\usepackage{enumitem}
\usepackage{graphicx}
\usepackage{tikz-cd, tikz}
\usepackage{float}
\usepackage[justification=centering]{caption}
\usepackage{subcaption}
\usepackage{wasysym}

\numberwithin{equation}{section}

\newtheorem{theorem}{Theorem}[section]
\newtheorem{proposition}[theorem]{Proposition}
\newtheorem{lemma}[theorem]{Lemma}
\newtheorem{corollary}[theorem]{Corollary}

\newtheorem*{theorem*}{Theorem}
\newtheorem*{proposition*}{Proposition}
\newtheorem*{lemma*}{Lemma}
\newtheorem*{corollary*}{Corollary}

\newtheorem{mthm}{Theorem} 

\theoremstyle{definition}
\newtheorem{definition}[theorem]{Definition}
\newtheorem{observation}[theorem]{Observation}
\newtheorem{remark}[theorem]{Remark}
\newtheorem{example}[theorem]{Example}
\newtheorem{question}[theorem]{Question}

\newtheorem{notation}[theorem]{Notation}

\newtheorem*{definition*}{Definition}
\newtheorem*{observation*}{Observation}
\newtheorem*{remark*}{Remark}
\newtheorem*{example*}{Example}
\newtheorem*{question*}{Question}
\newtheorem*{exercise*}{Exercise}
\newtheorem*{fact*}{Fact}
\newtheorem*{notation*}{Notation}

\newcommand{\Z}{\mathbb{Z}}

\DeclareMathOperator{\diam}{diam}

\newcounter{claimref}
\newcounter{claimcount}

\newenvironment{claim*}{\par\medskip\noindent\textbf{Claim:}\hspace{0.5mm}}{}

\newenvironment{claim*proof}{\medskip\noindent\emph{Proof of Claim.}\hspace{0.5mm}}
    {\leavevmode\unskip\penalty9999\hbox{}\nobreak\hfill\quad\hbox{$\diamondsuit$}\medskip}

\newcounter{subclaimcount}

\newenvironment{subclaim*}{\par\medskip\textbf{Subclaim:}\hspace{0.5mm}}{}
\newenvironment{subclaimproof*}{\medskip\noindent\emph{Proof of Subclaim.}\hspace{0.5mm}}
    {\leavevmode\unskip\penalty9999\hbox{}\nobreak\hfill\quad\hbox{$\fullmoon$}\medskip}

\crefname{cond}{condition}{conditions}
\creflabelformat{cond}{#2#1\@#3}
\crefname{obs}{observation}{observations}
\creflabelformat{obs}{#2#1\@#3}

\newcommand*{\eps}{\varepsilon}

\newcommand*{\R}{\mathbb{R}}

\newcommand*{\cQ}{\mathcal{Q}}
\newcommand*{\dist}{d}

\DeclareMathOperator{\type}{Typ}
\DeclareMathOperator{\sing}{sing}
\newcommand*{\sgen}[1]{\langle#1\rangle}
\newcommand*{\syl}[1]{|#1|_{\mathrm{syl}}}
\newcommand*{\bigast}{\mathop{\scalebox{1.5}{\raisebox{-0.2ex}{$*$}}}}
\renewcommand*{\subset}{\subseteq}
\renewcommand*{\geq}{\geqslant}
\renewcommand*{\leq}{\leqslant}
\renewcommand*{\ge}{\geqslant}
\renewcommand*{\le}{\leqslant}

\newcounter{shcount}
\newcounter{thmcount}

\swapnumbers
\newcommand*{\bsh}[1]{\theoremstyle{definition}\newtheorem{subhead\theshcount}[theorem]{#1}
    \begin{subhead\theshcount}} 
\newcommand*{\esh}{\end{subhead\theshcount}\stepcounter{shcount}} 
 
\newcommand*{\numberedtheorem}[3]{\theoremstyle{plain}\newtheorem*{makethm\thethmcount}{#1}
    \ifthenelse{\equal{#2}{}}{\begin{makethm\thethmcount}#3\end{makethm\thethmcount}\stepcounter{thmcount}}
    {\begin{makethm\thethmcount}[#2]#3\end{makethm\thethmcount}\stepcounter{thmcount}}} 

\makeatletter\newcommand{\linkdest}[1]{\Hy@raisedlink{\hypertarget{#1}{}}}\makeatother 

\usepackage{lmodern,anyfontsize} 
\usepackage{setspace} 
\usepackage{tocloft} 
\usepackage[nottoc]{tocbibind}

\usepackage{amsmath,amssymb,xfrac}
    \usepackage{enumitem} \setlist{nosep}
\usepackage{ifthen} 
\usepackage[pagebackref,breaklinks,unicode]{hyperref} 
    \renewcommand*{\backrefalt}[4]{\ifcase #1 (Not cited).\or (Cited p.~#2).\else (Cited pp.~#2).\fi} 
\usepackage{graphicx} 
\usepackage{mathrsfs}
\usepackage[T1]{fontenc}
\usepackage[margin=3cm]{geometry}
\usepackage{natbib}

\newcommand*{\ext}{\mathrm{ext}}
\newcommand*{\N}{\mathbb{N}}

\makeatletter
\newtheorem*{rep@theorem}{\rep@title}
\newcommand{\newreptheorem}[2]{%
\newenvironment{rep#1}[1]{%
 \def\rep@title{#2 \ref{##1}}%
 \begin{rep@theorem}}%
 {\end{rep@theorem}}}
\makeatother

\newreptheorem{theorem}{Theorem}
\newreptheorem{example}{Example}
\newreptheorem{fact}{Fact}

\newcommand*{\ssm}{\smallsetminus}

\definecolor{harrycomment}{rgb}{0.6,0,0.4}

\definecolor{shakedcomment}{rgb}{0, 0, 255}

\definecolor{oussamacomment}{rgb}{0,0.5,0}

\title{Quasiisometric embeddings of right-angled Artin groups into products}
\hypersetup{pdftitle = Quasiisometric embeddings of right-angled Artin groups into products}
\date{}
\author{Shaked Bader, Oussama Bensaid, and Harry Petyt}

\newcommand{\Addresses}{{\bigskip\footnotesize\par
\textsc{Mathematical Institute, University of Oxford, UK}\par\nopagebreak\textit{E-mail address}: 
\texttt{shaked.bader@sjc.ox.ac.uk}
\par\medskip\par
\textsc{Institut de Recherche en Mathématique et en Physique, Université Catholique de Louvain, Belgium} \par\nopagebreak\textit{E-mail address}: 
\texttt{oussama.bensaid@uclouvain.be}
\par\medskip\par
\textsc{Mathematics Institute, University of Warwick, UK}\par\nopagebreak\textit{E-mail address}: 
\texttt{harrypetyt@gmail.com}
}}

\begin{document}

\maketitle

\begin{abstract}
We provide obstructions to the existence of quasiisometric embeddings of right-angled Artin groups into products of trees, hyperbolic spaces, and other right-angled Artin groups.
\end{abstract}

{\hypersetup{hidelinks}\setcounter{tocdepth}{1}\tableofcontents\setcounter{tocdepth}{2}}

\section{Introduction} 

A \emph{right-angled Artin group} is a ``free partially abelian group'': it has a generating set where the only relations are that some generators commute. This can be encoded by a simplicial graph, with a vertex for each generator and an edge corresponding to commutation. The correspondence is bijective \cite{droms:isomorphisms}, so we can unambiguously pass between a graph $\Gamma$ and its associated right-angled Artin group $A_\Gamma$.

On one end of the spectrum are free groups, which correspond to edgeless graphs, and on the other are free-abelian groups, corresponding to complete graphs.
Between these simple extremes hides a large variety of examples with different behaviours. The class is closed under free products (disjoint unions of graphs) and direct products (joins of graphs). In this article we are interested in the large-scale geometry of right-angled Artin groups.

Quasiisometries provide a natural notion of isomorphism in the setting of large-scale geometry, and there has been a lot of work on the quasiisometric classification of right-angled Artin groups; see \cite{bestvinakleinersageev:asymptotic,huang:quasiisometric:1,huang:quasiisometric:2,huang:commensurability,margolis:quasiisometry}, for instance. More generally, it is natural to ask about quasiisometric embeddings (geometric monomorphisms) between right-angled Artin groups, and these were studied in \cite{baderbensaidpetyt:quasiisometric:rigidity,baderbensaidpetyt:quasiisometric:flexibility}. Because of examples like the log-spiral, when working with quasiisometric embeddings it is reasonable to assume that the domain and codomain have the same dimension, which is something that holds automatically for quasiisometries. 
(The \emph{dimension} of a right-angled Artin group $A_\Gamma$ is the maximal dimension of its $\mathbb{Z}^n$ subgroups, which is equal to the clique number of $\Gamma$.)

Amongst other things, in \cite[Thm~5.8]{baderbensaidpetyt:quasiisometric:rigidity} it was shown that 
if a product of $n$ free groups quasiisometrically embeds in an $n$--dimensional right-angled Artin group $A_\Gamma$, then that fact can already be seen in the graph $\Gamma$: it must contain an induced subgraph that is a join of $n$ edgeless graphs. Here we are interested in the reverse question: when does a right-angled Artin group quasiisometrically embed in a product?

In the positive direction, a beautiful result of Rull \cite[Thm~1.1]{rull:embedding} shows that if $\Gamma$ is $n$--colourable, then $A_\Gamma$ can be quasiisometrically embedded in $(F_2)^n$, the product of $n$ free groups of rank two.
Her proof uses the ``Alice's diary'' construction introduced for hyperbolic groups in \cite{buyalodranishnikovschroeder:embedding}. These methods have also been refined by Nairne to construct quasiisometric embeddings of toral relatively hyperbolic groups into products of free groups; see \cite{nairne:embedding}.

Our principal focus in this article is on finding obstructions to the existence of such embeddings. The main technical result we prove for this purpose is the following, which extends results of \cite{baderbensaidpetyt:quasiisometric:rigidity} to a larger family of graphs. See \cref{thm:beyond_FB_induced_embedding_singboundary}, \cref{rem:bg_asrk}, and \cref{rem:bg_hyp} for more general statements.

\begin{mthm} \label{mthm:bg_induced}
Let $\Gamma$ be a graph with clique number $n$, and let $Y$ be an $n$--dimensional CAT(0) cube complex. If $\Gamma$ is branch-graded and there is a quasiisometric embedding $A_\Gamma\to Y$, then there is a combinatorial embedding $\Gamma\to\partial_{\mathrm{sing}}Y$.
\end{mthm}

The \emph{branch-graded} assumption (\cref{def:bg}) is a condition on the intersection pattern of the maximal cliques in $\Gamma$ that covers many natural examples. When $n=2$, it is equivalent to $\Gamma$ having no leaves. For $n=3$ it includes most simplicial triangulations of compact surfaces (\cref{lem:boundary_sparse}), and in general it includes all simplicial triangulations of closed manifolds (\cite[Ex.~3.3]{baderbensaidpetyt:quasiisometric:rigidity}). The \emph{singular boundary graph} $\partial_{\mathrm{sing}}Y$ of a CAT(0) cube complex $Y$, introduced in \cite{baderbensaidpetyt:from}, is a subset of the Tits boundary that is sensitive to the cubical structure of $Y$; see \cref{def:singular_boundary_graph}.

We remark that in the two--dimensional setting, \cref{mthm:bg_induced} was already established in \cite[Thm~K]{baderbensaidpetyt:quasiisometric:rigidity} for all triangle-free graphs with no connected components that are stars. The novelty of \cref{mthm:bg_induced} is in the higher-dimensional setting.

Using \cref{mthm:bg_induced} we prove the following, which in particular shows that Rull's theorem is optimal in the two-dimensional case. See \cref{thm:fb_in_product_trees}.


\begin{mthm} \label{mthm:rull}
Let $\Gamma$ be a graph with clique number $n$. If $n>2$ then assume that $\Gamma$ is branch-graded. The following are equivalent. 
\begin{enumerate}
\item   $\Gamma$ is $n$--colourable.
\item   $A_\Gamma$ can be equivariantly isometrically embedded in a product of $n$ locally infinite trees. \label{mit:2}
\item   $A_\Gamma$ can be quasiisometrically embedded in $(F_2)^n$. \label{mit:3}
\item   $A_\Gamma$ can be quasiisometrically embedded in a product of $n$ hyperbolic spaces. 
\end{enumerate}
\end{mthm}

In particular, products of arbitrary hyperbolic spaces are no better as receivers of quasiisometric embeddings than products of trivalent trees, at least in the branch-graded case. Moreover, if a branch-graded graph is not perfect (its chromatic number is greater than its clique number), then its associated right-angled Artin group cannot be quasiisometrically embedded in a product of dimension-many trees.

In view of items \ref{mit:2} and \ref{mit:3} in \cref{mthm:rull}, it is natural to ask when (the standard Cayley graph of)a right-angled Artin group can be \emph{isometrically} embedded into a product of finite-rank free groups, or more generally a product of locally-finite trees. We show in \cref{cor:ie_prod_free} that this can only be done in the obvious cases.

\begin{mthm} \label{mthm:ie}
A right-angled Artin group can be isometrically embedded in a finite product of locally finite trees if and only if it is a direct product of finite-rank free groups.
\end{mthm}

We also consider \emph{stable} quasiisometric embeddings of right-angled Artin groups. A \emph{standard geodesic} in $A_\Gamma$ is just a coset of the cyclic subgroup generated by a single standard generator. A quasiisometric embedding between right-angled Artin groups is \emph{stable} if it maps standard geodesics within uniformly finite Hausdorff distance of standard geodesics. These embeddings are particularly nice, because they induce maps of \emph{extension graphs} \cite[Prop.~6.2]{baderbensaidpetyt:quasiisometric:rigidity}.
By the nature of the Alice's diary construction, Rull's embeddings are almost never stable \cite{rull:embedding}, and we show in \cref{cor:no_stable_Rull} that this is necessary.

\begin{mthm} \label{mthm:anti-rull}
A right-angled Artin group $A_\Gamma$ admits a stable quasiisometric embedding in a finite product of finite-rank free groups only if $A_\Gamma$ is a direct product of finite-rank free groups.
\end{mthm}

This can be viewed as an instance of a more general phenomenon, where right-angled Artin groups that are not ``product-like'' can only be quasiisometrically embedded in a product in a fairly degenerate way. The following is \cref{cor:stable_irreducible}.

\begin{mthm} \label{mthm:stable_irreducible}
Let $\Gamma$ be a connected graph with diameter at least four, and let $\Lambda$ be a finite graph. If $f:A_\Gamma\to A_\Lambda$ is a stable quasiisometric embedding, then $f$ induces a combinatorial embedding $\Gamma^\ext\to\Lambda_i^\ext$ for some irreducible $\Lambda_i$ in the join decomposition of $\Lambda$.
\end{mthm}

Here $\Gamma^\ext$ denotes the \emph{extension graph} of $\Gamma$, which is a graph capturing the relations between standard geodesics in $A_\Gamma$ \cite{kimkoberda:embedability}. The theorem \emph{almost} says that $f(A_\Gamma)$ is contained in a single direct factor of $A_\Lambda$, but it is not possible to conclude this in general; see \cref{rem:not_into_factor}.

It is noteworthy that Theorems~\ref{mthm:ie}, \ref{mthm:anti-rull}, and~\ref{mthm:stable_irreducible} do not require the domain and codomain to have the same dimension.

\subsection*{Outline of the article}

\begin{itemize}
\item   \cref{sec:prelims} covers background on CAT(0) cube complexes and right-angled Artin groups.
\item   The branch-graded condition is introduced in \cref{sec:beyond_FB_2}, where we also prove \cref{mthm:bg_induced}.
\item   \cref{sec:embeddings_into_trees} is about quasiisometric embeddings into products of trees. It uses the results of \cref{sec:beyond_FB_2} to prove \cref{mthm:rull}.
\item   In \cref{sec:isometric_embeddings} we consider isometric embeddings into products of locally-finite trees and prove \cref{mthm:ie}.
\item   Finally, \cref{antirull} is about stable quasiisometric embeddings. This is where Theorems~\ref{mthm:anti-rull} and~\ref{mthm:stable_irreducible} are proved.
\end{itemize}

\subsection*{Acknowledgements}
OB acknowledges support from the FWO and F.R.S.-FNRS under the Excellence of Science (EOS) programme (project ID 40007542). SB and HP thank the Isaac Newton Institute for their hospitality during the programme \emph{Operators, Graphs, Groups}, where some of the work on this paper took place (EPSRC grant EP/Z000580/1).
We also thank Jack Button for useful comments on an earlier version.

There was no use of AI in the production of this paper.

\section{Preliminaries} \label{sec:prelims}

\subsection{CAT(0) spaces}\label{sec:CAT(0)_spaces}

We shall use only a limited amount of CAT(0) geometry in this paper. We refer the reader to \cite{ballmann:lectures} and \cite{bridsonhaefliger:metric} for background on the subject. The main class of CAT(0) spaces we are concerned with in this paper are CAT(0) cube complexes (see \cref{subsec:ccc}) and more specifically, universal covers of Salvetti complexes. 

\begin{definition}[Flats, orthants]\label{def:flats_in_CAT(0)}
Let $(X,d)$ be a CAT(0) space. A \emph{$k$--flat} in $X$ is the image of an isometric embedding $(\mathbb R^k,\|\cdot\|_2)\to (X,d)$.
A \emph{$k$--orthant} in a CAT(0) space $X$ is the image of an isometric embedding of $([0,\infty)^k,\|\cdot\|_2)$ into $X$.
\end{definition}

\begin{definition}[Parallels] \label{def:parallel_set}
    Two flats $F$ and $F'$ in a complete CAT(0) space $X$ are \emph{parallel} if they are at finite Hausdorff distance from one another. Equivalently, the distance from $F$ to $F'$ and the distance from $F'$ to $F$ are constant. 
    %
\end{definition}

\begin{definition}[Tits boundary]
We denote by $\partial_TX$ the \emph{Tits boundary} of a CAT(0) space $X$, and equip it with the \emph{angular} metric. As a set, $\partial_TX$ is the set of equivalence classes of geodesic rays in $X$, where two rays are equivalent whenever they are at finite Hausdorff distance from one another. The angle between two elements of $\partial_TX$ is the supremal angle formed by two representative rays with the same basepoint.
\end{definition}

If a subset $A$ of a CAT(0) space lies at finite Hausdorff distance from a finite union of orthants, then we denote by $\partial_T A$ the union of the Tits boundaries of those orthants. 




\subsection{CAT(0) cube complexes} \label{subsec:ccc}

We refer the reader to \cite{wise:structure,bowditch:median:book,genevois:algebraic} for basic facts about CAT(0) cube complexes. 

A \emph{cube complex} is a cell complex obtained by gluing Euclidean unit cubes $[0,1]^n$ along isometric faces. A CAT(0) cube complex is a simply connected cube complex such that the link of each vertex is a flag simplicial complex. The \emph{dimension} of a CAT(0) cube complex is the supremal dimension of the cubes involved in its construction. 

A key combinatorial feature of CAT(0) cube complexes is their \emph{hyperplanes}. For each edge $e$ of $X$, the hyperplane \emph{dual} to $e$ can be defined as an equivalence class of edges: set $e_1\sim e_2$ whenever they are opposite edges in a square and extend transitively. Each hyperplane $h$ separates $X$ into two \emph{halfspaces}. We say that $h$ \emph{separates} two points of $X$ if they lie in different halfspaces. Two hyperplanes \emph{cross} if all four possible intersections of halfspaces are nonempty.

If $X$ is a CAT(0) cube complex, then the length metric $d_2$ obtained by equipping each cube with the $\ell^2$ metric makes $X$ into a CAT(0) space. If one instead equips the cubes with the $\ell^1$ metric, then one obtains a length metric $d_1$ on $X$. For vertices $x,y\in X$, the distance $d_1(x,y)$ is equal to the number of hyperplanes that separate $x$ from $y$. If $X$ is finite-dimensional, then $d_1$ and $d_2$ are bilipschitz equivalent. In this paper, we shall almost always consider CAT(0) cube complexes as being equipped with the metric $d_1$. The exception is when we are discussing flats.

\begin{definition}[Singular] \label{def:ccc_singular}
Let $X$ be a CAT(0) cube complex. A $k$--flat $F\subset X$ is \emph{singular} if it is contained in the $k$--skeleton of $X$. A \emph{singular geodesic} is a singular $1$--flat. A \emph{singular geodesic ray} is a subset of the 1--skeleton that is the image of an isometric embedding $[0,\infty)\to(X,d_2)$. 
\end{definition}


The following notion of boundary was introduced in \cite[Def.~11.1]{baderbensaidpetyt:from}. 

\begin{definition}[Singular boundary graph] \label{def:singular_boundary_graph}
Let $X$ be a finite-dimensional CAT(0) cube complex. The \emph{singular boundary graph} of $X$, denoted $\partial_{\mathrm{sing}}X$, is the following graph. The vertices of $\partial_{\mathrm{sing}}X$ are the equivalence classes of singular geodesic rays in $X$, where geodesic rays are deemed equivalent if they lie at finite Hausdorff distance from one another. Two vertices of $\partial_{\mathrm{sing}}X$ are joined by an edge whenever they admit singular geodesic representatives that span a 2--orthant of $X$. By \cite[Lem.~4.6]{baderbensaidpetyt:from}, this is equivalent to them having angle $\frac\pi2$ in $\partial_TX$.
\end{definition}

\begin{example}\label{ex:stairwase_and_sing_bound_product_trees}
If $T$ is a tree, then $\partial_{\mathrm{sing}} T$ is the edgeless graph whose vertex set is $\partial_T T$. More generally, if $X=T_1\times\cdots\times T_n$ is a product of trees, then $\partial_{\mathrm{sing}}X$ is the complete $n$-partite graph whose parts are the sets $\partial_T T_1,\dots,\partial_T T_n$. In particular, $\partial_{\mathrm{sing}}X$ is $n$--colourable.
\end{example}

See \cite[\S11]{baderbensaidpetyt:from} for more discussion of the singular boundary graph and a comparison to other notions of boundary for CAT(0) cube complexes.


\subsection{Right-angled Artin groups} \label{sec:raags}

A \emph{right-angled Artin group} (RAAG) is a group $A$ with a generating set $S$ such that the only relations are that some elements of $S$ commute. This can be encoded by the (simplicial) graph $\Gamma$ that has vertex set $S$ and an edge between vertices whose corresponding generators commute. We write $A=A_\Gamma$, and refer to $S$ as the \emph{standard} generating set of $A_\Gamma$. 

\begin{remark}
Unless otherwise specified, a right-angled Artin group need not be finitely generated.
\end{remark}

\bsh{The space $X_\Gamma$} \label{sh:the_space_X_Gamma}
The Cayley graph of $A_\Gamma$ with respect to the generating set $S$ is the 1--skeleton of a CAT(0) cube complex $X_\Gamma$. The quotient of $X_\Gamma$ by the action of $A_\Gamma$ is known as the \emph{Salvetti complex} of $A_\Gamma$. 

The dimension of $X_\Gamma$ is equal to the clique number of $\Gamma$, denoted $\omega_\Gamma$. If $\omega_\Gamma<\infty$, then $X_\Gamma$ is quasiisometric to the standard Cayley graph $A_\Gamma$, its 1--skeleton. Hence for most purposes we will not distinguish the two, and just write $A_\Gamma$. For instance, the \emph{dimension} of $A_\Gamma$ is the dimension of $X_\Gamma$. But when thinking of the space as a CAT(0) object we will consider $X_\Gamma$. 
\esh

In view of the above, when making a metric statement about a right-angled Artin group $A_\Gamma$ where $\Gamma$ is infinite, we are referring to the standard Cayley graph of $A_\Gamma$.


\begin{definition}[Standard]\label{def:standard}
Each edge of $X_\Gamma$ is labelled by some standard generator in $S$. A geodesic in the 1--skeleton of $X_\Gamma$ is \emph{standard} if all its edges have the same label. Identifying with $A_\Gamma$, we denote the standard geodesic labelled by $s$ and going through $g$ by $g\sgen s$. 

For $k>1$, a $k$--flat in $X_\Gamma$ is \emph{standard} if it is singular and contains $k$ pairwise non-parallel standard geodesics. Standard flats correspond to cliques in $\Gamma$.
\end{definition}

We build on the branching conditions introduced in \cite[\S3]{baderbensaidpetyt:quasiisometric:rigidity}. Recall that a clique in a graph is \emph{maximal} if it is not contained in any larger clique.

\begin{definition}[(Directionally) branch-complemented] \label{def:fully_branching_in_graphs}
Let $\Gamma$ be a (simplicial) graph with clique number $n \geq 2$.
\begin{enumerate}
\item A clique is said to be \emph{branching} if it is the intersection of some maximal cliques, one of which is an $n$--clique.
\item A $k$--clique $W$ is called \emph{branch-complemented} if it is contained in an $n$--clique $K$ such that both $\{W\}$ and $K\ssm\{W\}$ are branching. 
We call the clique $K\ssm W$ a \emph{complementary clique} of $W$.
\item A subgraph of $\Gamma$ is \emph{directionally branch-complemented} if each of its vertices is branch-complemented.
\end{enumerate}
\end{definition}




\subsection{Asymptotic cones}

Asymptotic cones provide a way to view a space ``from infinity'', which removes coarseness. They will be used as a technical tool in \cref{sec:beyond_FB_2}. We refer the reader to \cite{vandendrieswilkie:gromov} for a good introductory discussion.

Let $\omega$ be a nonprincipal ultrafilter on $\N$. Given a sequence $(x_n)$ in $\R$, if there is some $x\in\R$ such that $\{n\,:\,|x_n-x|<\eps\}\in\omega$ for all $\eps>0$, then we write $x=\lim_\omega(x_n)$.

\begin{definition}[Asymptotic cone] \label{def:asymptotic_cone}
Let $(X,\dist)$ be a metric space and fix a basepoint $o\in X$. Let $(\lambda_n)$ be a divergent sequence of positive numbers. Let $Y$ be the set of all sequences $(x_n)$ in $X$ such that $\lim_\omega(\frac1{\lambda_n}\dist(x_n,o))$ exists. The function $\hat\dist((x_n),y_n))=\lim_\omega(\frac1{\lambda_n}\dist(x_n,y_n))$ is a pseudometric on $Y$. The \emph{asymptotic cone} $\lim_\omega(X,(\lambda_n),o)$ of $X$ is the metric quotient of $Y$. We often abuse notation and simply write $\hat X=\lim_\omega(X,(\lambda_n),o)$.

The \emph{ultralimit} of $A\subset X$ is the subset $\hat A\subset\hat X$ consisting of all points that can be represented by a sequence $(a_n)$ with $a_n\in A$ for all $n$. The \emph{cone point} of $\hat X$ is the point $\hat o$.
\end{definition}

If $f:X\to Y$ is a quasiisometric embedding of metric spaces, then for every $o\in X$ and every divergent sequence $(\lambda_n)$ of positive numbers there is an induced biLipschitz embedding $\hat f:\hat X\to\hat Y$, where $\hat X=\lim_\omega(X,(\lambda_n),o)$ and $\hat Y=\lim_\omega(Y,(\lambda_n),f(o))$.

The following technical proposition specialises results from \cite{baderbensaidpetyt:from} to standard flats in RAAGs. The \emph{asymptotic rank} of a CAT(0) cube complex $Y$ is the supremal $n$ such that $Y$ contains an isometric copy of $[0,m]^n$ for all $m\in\mathbb N$ (see \cite[Prop.~2.18]{baderbensaidpetyt:from}). It is clearly bounded above by the dimension of $Y$, which makes the following more general than assuming that $Y$ has dimension $n$. 

\begin{proposition}\label{prop:branch_complemented_splitting_cone}
Let $\Gamma$ be a graph with clique number $n$ and let $Y$ be a finite-dimensional CAT(0) cube complex with asymptotic rank $n$. Let $\hat f:\widehat{X_\Gamma}\to\hat Y$ be a biLipschitz embedding of asymptotic cones. Suppose that $C=B\sqcup K\subseteq\Gamma$ is a clique such that $B\ne\varnothing$ is directionally branch-complemented and $K\ne\varnothing$ is branching. Write $F_C,F_B,F_K\subseteq X_\Gamma$ for the corresponding standard flats through the identity. 
    
    The image $\hat f(\widehat{F_B})$ is a singular $|B|$--flat, and $\hat f(\widehat{F_C})$ splits isometrically as
    $$ \hat f(\widehat{F_C}) \,=\, \hat f(\widehat{F_B})\times\hat f(\widehat{F_K}).$$
    In particular, if $O_B\subseteq\hat f(\widehat{F_B})$ and $O_K\subseteq\hat f(\widehat{F_K})$ are singular $p$-- and $q$-- orthants based at the cone point of $\hat Y$, then their axis-rays together span a singular $(p+q)$--orthant in $\hat f(\widehat{F_C})$.
\end{proposition}

\begin{proof}
By \cite[Lem.~3.6(1)]{baderbensaidpetyt:quasiisometric:rigidity}, there exists a graph $\Lambda_B$, a $|B|$--clique $B'\subset\Lambda_B$, a $2^{n^2}$--quasiisometry $\iota_B:X_{\Lambda_B}\to X_\Gamma$, and a standard flat $F'_B\subset X_{\Lambda_B}$ such that $\iota_B(F'_B)$ lies at Hausdorff distance at most $2^{n^2}$ from $F_B$ and $F'_B$ is directionally branch-complemented in the sense of \cite[Def.~9.5]{baderbensaidpetyt:from}. Passing to asymptotic cones, we have $\widehat\iota_B(\widehat{F'_B})=\widehat{F_B}$. By \cite[Prop.~9.20]{baderbensaidpetyt:from}, it follows that $\hat f(\widehat{F_B})$ is a singular flat.

Similarly, since $K$ is branching, applying \cite[Lem.~3.6(3)]{baderbensaidpetyt:quasiisometric:rigidity} gives a graph $\Lambda_K$, a $|K|$--clique $K'\subset\Lambda_K$, a $2^n$--quasiisometry $\iota_K:X_{\Lambda_K}\to X_\Gamma$, and a standard flat $F'_K\subset X_{\Lambda_K}$ such that $\iota_K(F'_K)$ lies at Hausdorff distance at most $2^n$ from $F_K$ and $F'_K$ is an intersection of $n$--flats in $X_{\Lambda_K}$. We again have $\widehat{\iota_K}(\widehat{F'_K})=\widehat{F_K}$. It thus follows from \cite[Thm~6.13]{baderbensaidpetyt:from} that $\hat f(\widehat{F_K})$ is a finite union of singular $|K|$--orthants. 

Therefore, \cite[Prop.~7.9]{baderbensaidpetyt:from} applies to $\hat f|_{\widehat{F_C}}$, and gives the desired splitting. The last assertion follows by taking the product of the orthants.
\end{proof}

\section{Branch-graded subgraphs} \label{sec:beyond_FB_2}

The main theorem of this section, \cref{thm:beyond_FB_induced_embedding_singboundary}, extends the boundary obstruction of \cite[Thm~3.10]{baderbensaidpetyt:quasiisometric:rigidity} to a larger class of subgraphs. From a quasiisometric embedding $A_{\Gamma}\to A_{\Lambda}$, \cite[Thm~3.10]{baderbensaidpetyt:quasiisometric:rigidity} obtains a canonical graph embedding $\Gamma\to \partial_{\sing}A_\Lambda$. In \cref{thm:beyond_FB_induced_embedding_singboundary}, we construct a non-canonical embedding.
\medskip


\begin{definition}[Complementary degree]
Let $\Lambda$ be a graph with clique number $n$, and let $v$ be a vertex of $\Lambda$. Let $\mathcal C_\Lambda(v)$ denote the collection of branch-complemented $(n-1)$--cliques of $\Lambda$ that contain $v$. 

We define the \emph{complementary degree} $d_{\mathrm{comp}}(v)$ of $v$ as follows. If $v$ is branch-complemented, then set $d_{\mathrm{comp}}(v)=0$ and let $K_\Lambda(v)=v$. If $v$ is not branch-complemented and $\mathcal C_\Lambda(v)=\varnothing$, then set $d_{\mathrm{comp}}(v)=\infty$. Otherwise, define
\[
K_\Lambda(v) \,= \bigcap_{C\in\mathcal C_\Lambda(v)} C
\]
and let $d_{\mathrm{comp}}(v)=|K_\Lambda(v)|$. See Figure~\ref{fig:comp_degree_exples}.
\end{definition}


\begin{figure}[htbp]
\centering\vspace{-3mm}
\begin{tikzpicture}[baseline=-.1cm, scale=0.7]
\draw[thick]
(0,0) -- (0,1.8) -- (0,3.2)
(0,0) --(-1.5,-1.4) -- (-1.5,0.5) -- (0,0) -- (1.5,0.5) -- (1.5,-1.4)
(-1.5,0.5) -- (0,1.8) -- (1.5,0.5)
(0,0) -- (1.5,-1.4) ; 

\draw[thick,red] (-1.5,0.5) -- (0,0) -- (1.5,0.5);

\fill[blue] (0,0) circle(.18);
\fill[red] (0,1.8) circle(.18);
\fill (0,3.2) circle(.18);
\fill[green!50!black] (1.5,0.5) circle(.18);
\fill[green!50!black] (-1.5,0.5) circle(.18);
\fill (-1.5,-1.4) circle(.18);
\fill (1.5,-1.4) circle(.18);
\end{tikzpicture}
\caption{The red vertex is branch-complemented, i.e.\ of complementary degree $0$. The red edges are the branch-complemented $2$-cliques. The blue vertex has complementary degree $1$, and the green vertices have complementary degree $2$.}
\label{fig:comp_degree_exples}

\end{figure}
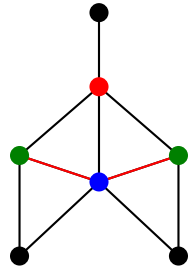

The following definition is rooted in the idea of leveraging the strong information we have about branch-complemented cliques in $\Lambda$ to understand more general cliques.

\begin{definition}[Branch-graded] \label{def:bg}
    Let $\Lambda$ be a graph with clique number $n$, and let $\Gamma\leq \Lambda$ be an induced subgraph. We say that $\Gamma$ is \emph{branch-graded} if the following hold in $\Lambda$.
    \begin{enumerate}
        \item \label{first_item} If $v\in\Gamma$ is not branch-complemented, then $v$ lies in a branch-complemented $(n-1)$--clique of $\Lambda$.
        \item \label{second_item} $K_\Lambda(v)\subseteq\Gamma$ for all $v\in\Gamma$. 
        \item For every edge $\{a,b\}$ of $\Gamma$, one of the following holds in $\Lambda$: \label{third_item}
        \begin{enumerate}
        \item \label{it:edge_branch_complemented} $a$ and $b$ are branch-complemented;
        \item\label{it:edge_complementary}
        $a$ is branch-complemented and $b$ lies in a complementary clique of $a$, or vice versa;
        \item\label{it:edge_Kcomp}
        $a\in K_\Lambda(b)$ or $b\in K_\Lambda(a)$.
        \end{enumerate}
    \end{enumerate}
\end{definition}

In particular, if the subgraph $\Gamma \leq \Lambda$ is directionally branch-complemented, then $\Gamma$ is branch-graded in $\Lambda$. More examples are provided in Figure~\ref{fig:exples_branch_graded}. If $\Lambda$ is a graph that is branch-graded as an induced subgraph of itself, then we simply refer to $\Lambda$ as being branch-graded. We also refer to $A_\Lambda$ as being branch-graded in this case.

\begin{figure}[htbp]
\centering
\begin{tikzpicture}[baseline=-.1cm, scale=0.6]
\draw[thick]
(0,0) -- (0,1.8) -- (0,3.2)
(0,0) --(-1.5,-1.4) -- (-1.5,0.5) -- (0,0) -- (1.5,0.5) -- (1.5,-1.4)
(-1.5,0.5) -- (0,1.8) -- (1.5,0.5)
(0,0) -- (1.5,-1.4)
(-1.5,-1.4) -- (1.5,-1.4)
(1.5,-1.4) -- (2.5,-2.3)
(-1.5,-1.4) -- (-2.5,-2.3);

\draw[thick,orange] (0,0) -- (0,1.8);

\draw[thick,orange] (0,0) --(-1.5,-1.4) -- (-1.5,0.5) -- (0,0) -- (1.5,0.5) -- (1.5,-1.4)
(-1.5,0.5) -- (0,1.8) -- (1.5,0.5) ;
\draw[thick,orange] (0,0) -- (1.5,-1.4) -- (-1.5,-1.4);

\draw[thick,orange] (-1.5,0.5) -- (0,0) -- (1.5,0.5);

\fill[blue] (0,0) circle(.18);
\fill[red] (0,1.8) circle(.18);
\fill (0,3.2) circle(.18);
\fill[green!50!black] (1.5,0.5) circle(.18);
\fill[green!50!black] (-1.5,0.5) circle(.18);
\fill[red] (-1.5,-1.4) circle(.18);
\fill[red] (1.5,-1.4) circle(.18);
\fill (-2.5,-2.3) circle(.18);
\fill (2.5,-2.3) circle(.18);
\end{tikzpicture}
\hspace{1cm}
\begin{tikzpicture}[baseline=-.1cm, scale=0.6]
\draw[thick]
(0,0) -- (0,1.8) -- (0,3.2)
(0,0) --(-1.5,-1.4) -- (-1.5,0.5) -- (0,0) -- (1.5,0.5) -- (1.5,-1.4)
(-1.5,0.5) -- (0,1.8) -- (1.5,0.5)
(0,0) -- (1.5,-1.4)
(-1.5,-1.4) -- (1.5,-1.4) -- (0,-2.6) -- (-1.5,-1.4) ;

\draw[thick,orange] (0,0) -- (0,1.8);

\draw[thick,orange] (0,0) --(-1.5,-1.4) -- (-1.5,0.5) -- (0,0) -- (1.5,0.5) -- (1.5,-1.4)
(-1.5,0.5) -- (0,1.8) -- (1.5,0.5) ;
\draw[thick,orange] (0,0) -- (1.5,-1.4) -- (-1.5,-1.4);

\draw[thick,orange] (-1.5,0.5) -- (0,0) -- (1.5,0.5);

\fill[red] (0,0) circle(.18);
\fill[red] (0,1.8) circle(.18);
\fill (0,3.2) circle(.18);
\fill[green!50!black] (1.5,0.5) circle(.18);
\fill[green!50!black] (-1.5,0.5) circle(.18);
\fill[red] (-1.5,-1.4) circle(.18);
\fill[red] (1.5,-1.4) circle(.18);
\fill (0,-2.6) circle(.18);
\end{tikzpicture}
\hspace{1.2cm}
\begin{tikzpicture}[baseline=-.1cm, scale=0.6]
\draw[thick,orange]
(1.6,-1.8) -- (0.8,0) -- (1.6,1.8) -- (1.6,-1.8)
(-1.6,-1.8) -- (-0.8,0) --(-1.6,1.8) -- (-1.6,-1.8)
(1.6,-1.8) -- (-1.6,-1.8)
(0.8,0) -- (-0.8,0)
(1.6,1.8) -- (0,2.4) -- (-1.6,1.8) 
(1.6,1.8) -- (0,1.2) -- (-1.6,1.8)
(0,2.4) -- (0,1.2)
(-0.8,0) -- (1.6,-1.8)  ;

\fill[red] (1.6,-1.8)  circle(.18);
\fill[green!50!black] (-1.6,-1.8)  circle(.18);
\fill[green!50!black] (0.8,0) circle(.18);
\fill[red] (-0.8,0) circle(.18);
\fill[red] (1.6,1.8) circle(.18);
\fill[red] (-1.6,1.8) circle(.18);
\fill[green!50!black] (0,1.2) circle(.18);
\fill[green!50!black] (0,2.4) circle(.18);
\end{tikzpicture}

\vspace{0.8cm}

\begin{tikzpicture}[baseline=-.1cm, scale=0.6]
\draw[thick,orange]
(0.8,0) -- (2.4,0) -- (0,1.8) -- (0.8,0)
(-0.8,0) -- (-2.4,0) -- (0,1.8) -- (-0.8,0)
(2.4,0) -- (1.6,-1.8) -- (0.8,0)
(-2.4,0) -- (-1.6,-1.8) -- (-0.8,0)
(1.6,-1.8) -- (-1.6,-1.8);

\fill[red] (1.6,-1.8)  circle(.18);
\fill[red] (0,1.8)  circle(.18);
\fill[red] (-1.6,-1.8)  circle(.18);
\fill[green!50!black] (0.8,0) circle(.18);
\fill[green!50!black] (2.4,0) circle(.18);
\fill[green!50!black] (-2.4,0) circle(.18);
\fill[green!50!black] (-0.8,0) circle(.18);
\end{tikzpicture}
\hspace{1.2cm}
\begin{tikzpicture}[baseline=-.1cm, scale=0.6]
\draw[thick,orange]
(0.8,0) -- (2.4,0) -- (1.6,1.8) -- (0.8,0)
(-0.8,0) -- (-2.4,0) -- (-1.6,1.8) -- (-0.8,0)
(2.4,0) -- (1.6,-1.8) -- (0.8,0)
(-2.4,0) -- (-1.6,-1.8) -- (-0.8,0)
(1.6,1.8) -- (-1.6,1.8) 
(1.6,-1.8) -- (-1.6,-1.8);

\fill[red] (1.6,1.8)  circle(.18);
\fill[red] (1.6,-1.8)  circle(.18);
\fill[red] (-1.6,1.8)  circle(.18);
\fill[red] (-1.6,-1.8)  circle(.18);
\fill[green!50!black] (0.8,0) circle(.18);
\fill[green!50!black] (2.4,0) circle(.18);
\fill[green!50!black] (-2.4,0) circle(.18);
\fill[green!50!black] (-0.8,0) circle(.18);
\end{tikzpicture}
\hspace{1.2cm}
\begin{tikzpicture}[baseline=-.1cm, scale=0.6]
\draw[thick,orange]
(0,0) -- (1.5,1.5) 
(0,0) -- (1.5,-1.5) 
(0,0) -- (-1.5,1.5)
(0,0) -- (-1.5,-1.5)
(-1.5,-1.5) -- (-1.5,1.5) -- (1.5,1.5) -- (1.5,-1.5) -- (-1.5,-1.5);
\draw[thick]
(1.5,1.5) -- (2.6,-2.6) -- (-1.5,-1.5)  ;

\fill[blue] (0,0) circle(.18);
\fill[red] (1.5,1.5)  circle(.18);
\fill[red] (-1.5,-1.5)  circle(.18);
\fill[green!50!black] (1.5,-1.5) circle(.18);
\fill[green!50!black] (-1.5,1.5) circle(.18);
\fill (2.6,-2.6) circle(.18);
\end{tikzpicture}

\vspace{0.5cm}

\begin{tikzpicture}[baseline=-.1cm, scale=0.5]
\draw[thick,orange]
(0,0) -- (9,0) 
(1.5,2) -- (7.5,2)
(0,0) -- (1.5,-2) -- (3,0)
(0,0) -- (1.5,2) -- (3,0) -- (4.5,2) -- (6,0) -- (7.5,2) -- (9,0) ;

\draw[thick]
(1.5,-2) -- (-1.5,-2) -- (0,0)
(7.5,2) -- (10.5,2) -- (9,4) -- (7.5,2) 
(10.5,2) -- (9,0);

\fill[red] (0,0)  circle(.18);
\fill[red] (3,0)  circle(.18);
\fill[red] (4.5,2)  circle(.18);
\fill[red] (6,0)  circle(.18);
\fill[red] (7.5,2)  circle(.18);
\fill[green!50!black] (1.5,2)  circle(.18);
\fill[green!50!black] (1.5,-2)  circle(.18);
\fill[green!50!black] (9,0)  circle(.18);
\fill (-1.5,-2)  circle(.18);
\fill (10.5,2)  circle(.18);
\fill (9,4)  circle(.18);

\end{tikzpicture}

\caption{The orange graphs/subgraphs are branch-graded. The red, blue and green vertices have complementary degrees $0,1,2$, respectively.}
\label{fig:exples_branch_graded}

\end{figure}
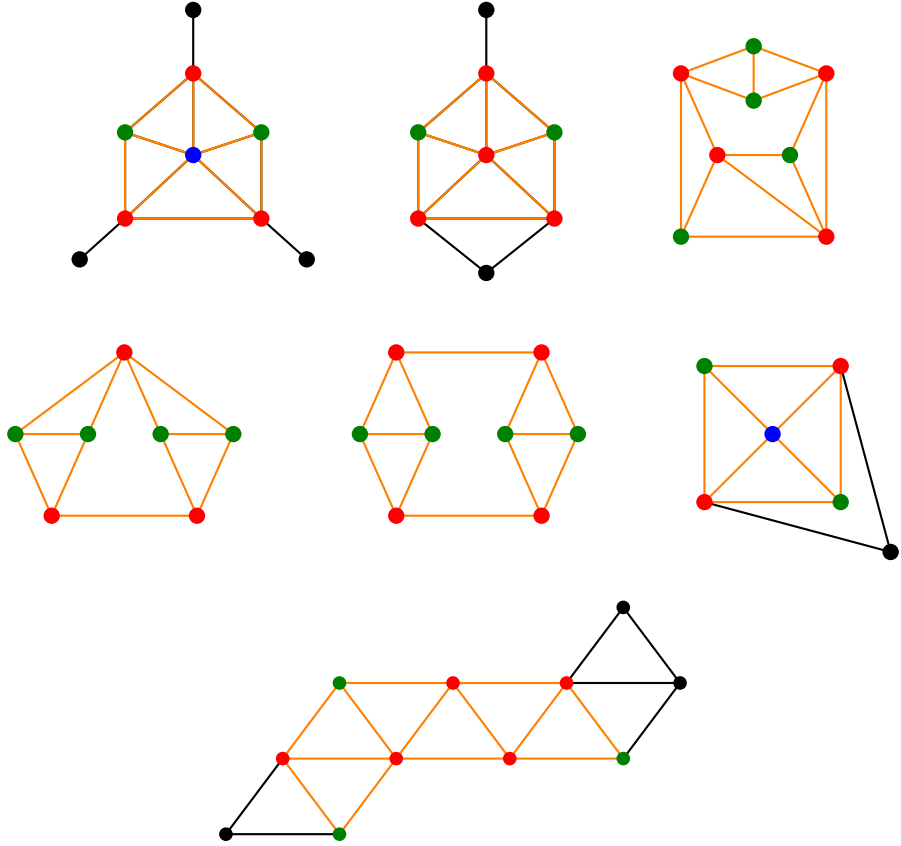


The following is the main result of this section. Recall that we always consider the metric on $A_\Lambda$ to be the one from the standard generating set, so the following also makes sense when $\Lambda$ is infinite.

\begin{theorem}\label{thm:beyond_FB_induced_embedding_singboundary}
Let $\Lambda$ be a graph with clique number $n$, and let $\Gamma$ be a branch-graded induced subgraph.
    If $f:A_\Lambda\to Y$ is a quasiisometric embedding into an $n$--dimensional CAT(0) cube complex, then there exists a combinatorial embedding $g:\Gamma\to\partial_{\mathrm{sing}}Y$.
\end{theorem}

The proof will go via asymptotic cones. Let $\hat Y$ be an asymptotic cone of $Y$ with basepoint $f(1)$. We will construct a map $\widetilde f:\Gamma\to \partial_{\mathrm{sing}}\hat Y $ whose image lies in the image of the isometric embedding $\partial_{\mathrm{sing}}Y \hookrightarrow \partial_{\mathrm{sing}}\hat Y $, given in \cite[Lem.~2.10]{baderbensaidpetyt:from}. Let us first give an outline of the construction of $\widetilde{f}$. In the proof below, we will prove that $\widetilde{f}$ gives rise to the desired embedding.

\begin{notation}
Both here and in the proof, we shall use the following notation.
\begin{itemize}
    \item We denote by $\widehat{X_\Lambda}$ an asymptotic cone of $X_\Lambda$ with basepoint 1 and by $\hat A\subset \widehat{X_\Lambda}$ the ultralimit of a subset $A\subseteq X_\Lambda$. Let
    $\hat f:\widehat{X_\Lambda} \to \hat Y$ be the induced biLipschitz embedding.
    \item For a clique $K\subseteq\Lambda$, let $F_K\subset X_\Lambda$ denote the standard flat through the identity associated with $K$. For a vertex $v$, let $\gamma_v$ be the standard geodesic through the identity labelled by $v$, and let $\gamma_v^+$ be its positive ray. Note that $\widehat{\gamma_v^+}=\widehat{\gamma_v}^+$.
    \item For $v\in\Lambda$, we shall write $K(v)=K_{\Lambda}(v)$, as all branching, branch complementation, and complementarity conditions are taken in $\Lambda$, even if $v$ happens to be a vertex of $\Gamma$.
\end{itemize}
\end{notation}

\begin{observation} \label{obs:comp_degree}
By \cref{first_item} in the definition of branch-graded, the clique $K(v)$ is defined for every $v\in\Gamma$. If $v$ is not branch-complemented, then $K(v)$ is a non-empty branching clique containing $v$, and $d_{\mathrm{comp}}(v)= |K(v)|\leq n-1$. If $u\in K(v)$, then every branch-complemented clique containing $v$ contains $u$, so $K(u)\subseteq K(v)$. Thus $d_{\mathrm{comp}}(u)\leq d_{\mathrm{comp}}(v)$, and if $d_{\mathrm{comp}}(u)=d_{\mathrm{comp}}(v)$ then $K(u)=K(v)$.
\end{observation}
    
    We construct the map $\widetilde f:\Gamma\to\partial_{\mathrm{sing}}\hat Y$ by assigning to each vertex $v \in \Gamma$ a singular geodesic ray $\rho_v \in \hat f(\widehat{F_{K(v)}})$, such that for each clique the chosen rays span a singular orthant, and setting $\widetilde f(v)$ to be the endpoint of $\rho_v$. We shall also ensure that $\rho_v$ is the ultralimit of some singular geodesic ray $\tau_v\subset Y$, and let $g(v)$ be the endpoint of $\tau_v$. 
    
    First, if $v$ is branch-complemented, then we show that $\hat f(\widehat{\gamma_v})$ is a singular geodesic and set $\rho_v=\hat f(\widehat{\gamma_v^+})$. 
    We then proceed by induction on $d_{\mathrm{comp}}$. 
    %
    %
    %
    For the induction step, the singular geodesic rays we choose in $\hat Y$ will not necessarily be canonically determined. For $k>0$, given a vertex $v$ with $d_{\mathrm{comp}}(v)=k$, we know that 
    %
    $K(v)$ consists of vertices of complementary degree at most $k$  (\cref{obs:comp_degree}). We shall define $\widetilde f$ on all vertices in $K(v)$ of complementary degree $k$ simultaneously. 
    By induction, the vertices of complementary degree less than $k$ were already assigned singular geodesic rays that span a singular orthant in $\hat f(\widehat{F_{K(v)}})$. We extend this singular orthant to a singular $k$--orthant in $\hat f(\widehat{F_{K(v)}})$, then assign its remaining axis-rays bijectively to the vertices of $K(v)$ of complementary degree $k$. 
    At each step of the induction, we check that $\widetilde f$ is injective, and use \Cref{prop:branch_complemented_splitting_cone} to check that for each clique the chosen rays span a singular orthant.

\begin{proof}[Proof of \Cref{thm:beyond_FB_induced_embedding_singboundary}] 
We shall use induction on $d_{\mathrm{comp}}$ to define, for a vertex $v\in\Gamma$, a singular ray $\rho_v\subset\hat f\left(\widehat{F_{K(v)}}\right)$ based at the cone point of $\hat Y$. We then set $\widetilde f(v)$ to be the point of $\partial_{\mathrm{sing}}\hat Y$ represented by $\rho_v$. We do this such that:
\begin{itemize}
    \item $\widetilde f$ is injective;
    \item (Cliques to orthants) for each clique $C\subseteq \Gamma$, the rays $\{\rho_v\}_{v\in C}$ span an orthant in $\hat f(\widehat{F_C})$;
    \item there is a singular geodesic ray $\tau_v\subset Y$ whose ultralimit is $\rho_v$.
\end{itemize}

Once we have done this, we have an injective map $\widetilde f:\Gamma\to\partial_{\mathrm{sing}}\hat Y$, and the cliques to orthants condition implies that it is a combinatorial embedding. Define $g:\Gamma\to\partial_{\mathrm{sing}}Y$ by setting $g(v)$ to be the point represented by $\tau_v$. According to \cite[Lem.~2.10]{baderbensaidpetyt:from}, we have an isometric embedding $\partial_TY\to\partial_T\hat Y$, and it restricts to an isometric embedding $\varphi:\partial_{\mathrm{sing}}Y\to\partial_{\mathrm{sing}}\hat Y$. Since the ultralimit of $\tau_v$ is $\rho_v$ for every $v\in\Gamma$, we have $\widetilde f=\varphi g$, so $g$ is a combinatorial embedding.

It remains to make the above inductive construction.

\medskip\noindent
\textbf{Complementary degree 0.}
If $v\in\Gamma$ has complementary degree zero then it is branch-complemented. Let $\gamma_v=F_{K(v)}$ be the corresponding geodesic.
As in the proof of \cref{prop:branch_complemented_splitting_cone}, applying \cite[Lem.~3.6(1)]{baderbensaidpetyt:quasiisometric:rigidity} gives a graph $\Lambda'$ and a standard geodesic $\gamma'\subset X_{\Lambda'}$ that is branch-complemented in the sense of \cite[Def.~9.3]{baderbensaidpetyt:from} and has the property that $f(\gamma_v)$ lies at finite Hausdorff distance from the image of $\gamma'$ under a quasiisometric embedding $X_{\Lambda'}\to Y$. 

By applying Huang's quasiflats theorem \cite[Thm~1.1]{huang:top} and \cite[Prop.~5.7]{baderbensaidpetyt:from} together to $\gamma'$, we deduce that $f(\gamma_v)$ lies at finite Hausdorff distance from a finite union of singular geodesic rays.
Hence there is a singular geodesic ray $\tau_v\subset Y$ at finite Hausdorff distance from $f(\gamma_v^+)$. Let $\rho_v=\hat f(\widehat{\gamma_v}^+)$ be the singular geodesic ray that is the ultralimit of $\tau_v$. It is based at the cone point of $\hat Y$. Define $\widetilde f(v)\in\partial_{\mathrm{sing}}\hat Y$ to be its endpoint.


\medskip\noindent
\emph{Injectivity:} 
If $v$ and $w$ are distinct branch-complemented vertices, then $\widehat{\gamma_v}^+$ and $\widehat{\gamma_w}^+$ are not at finite Hausdorff distance in $\widehat{X_\Lambda}$. Since $\hat f$ is a biLipschitz embedding, their images $\rho_v$ and $\rho_w$ are not at finite Hausdorff distance in $\hat Y$. Hence $\widetilde f(v)\neq\widetilde f(w)$.

\medskip\noindent
\emph{Cliques to orthants:}  
Suppose that $C=\{v_1,\dots,v_p\}\subseteq\Gamma$ is a clique consisting of branch-complemented vertices. Similarly to before, we can use \cite[Lem.~3.6(1)]{baderbensaidpetyt:quasiisometric:rigidity} to find that $f(F_C)$ lies at finite Hausdorff distance from the quasiisometrically embedded image of a directionally branch-complemented (in the sense of \cite[Def.~9.5]{baderbensaidpetyt:from}) flat $F'$ in some $n$--dimensional CAT(0) cube complex. Applying \cite[Cor.~9.21]{baderbensaidpetyt:from} to $F'$, we deduce that $\hat f(\widehat{F_C})$ is a singular $p$--flat in $\hat Y$.
In particular, the chosen singular rays $\rho_{v_1},\dots,\rho_{v_p}$ span a singular $p$--orthant contained in $\hat f(\widehat{F_C})$.

\medskip\noindent
\textbf{Induction step.}
Let $k>0$ and assume that $\widetilde f$ has been defined on all vertices of complementary degree at most $k-1$, satisfying the above bullet points.

Let $v\in\Gamma$ be a vertex with $d_{\mathrm{comp}}(v)=k$. 
By \cref{obs:comp_degree}, we can write $K(v) = K_1 \cup K_2$, where $K_1$ is the set of vertices of complementary degree $k$, and $K_2$ is the set of vertices of complementary degree at most $k-1$.
By \cref{second_item} in the definition of branch-graded, $K(v)\subseteq\Gamma$, so $\widetilde f$ has been defined on $K_2$, but not on any vertex in $K_1$. 
For every $u\in K_1$, \cref{obs:comp_degree} also tells us that $K(u)=K(v)$. We will define the images under $\widetilde f$ of all vertices of $K_1$ simultaneously.

Just as before, since $K(v)$ is branching we can apply \cite[Lem.~3.6(3)]{baderbensaidpetyt:quasiisometric:rigidity}, Huang's quasiflats theorem \cite[Thm~1.1]{huang:top}, and \cite[Prop.~5.7]{baderbensaidpetyt:from} to see that $f(F_{K(v)})$ lies at finite Hausdorff distance from a finite union $U$ of singular $k$--orthants in $Y$. By \cite[Prop.~5.3]{baderbensaidpetyt:from}, the spherical complex $\partial_T U$ is homeomorphic to $\mathbb S^{k-1}$. Since the ultralimit of a singular orthant of $Y$ is a singular orthant of $\hat Y$ based at the cone point, it follows from \cite[Lem.~2.10]{baderbensaidpetyt:from} that $\partial_T\hat f(\widehat{F_{K(v)}})$ is isometric to $\partial_TU$, and in particular homeomorphic to $\mathbb S^{k-1}$.

By induction, since $K_2$ is a clique the rays $\{\rho_u\}_{u\in K_2}$ span a singular orthant of $\hat f(\widehat{F_{K_2}}) \subseteq \hat f(\widehat{F_{K(v)}})$. Its boundary is a simplex $\sigma_2$ of the spherical complex $\partial_T \hat f \left(\widehat{F_{K(v)}} \right)$. The latter is homeomorphic to $\mathbb S^{k-1}$, so $\sigma_2$ is contained in some $(k-1)$--simplex $\sigma \subseteq \partial_T\hat f(\widehat{F_{K(v)}})$. Let $O\subset U$ be the $k$--orthant corresponding to $\sigma$. We assign the vertices of $\sigma\ssm\sigma_2$ bijectively to the vertices of $K_1$. For each $w\in K_1$, let $\rho_w$ be the corresponding singular axis ray based at the cone point of $\hat Y$, and define $\widetilde f(w)\in\sigma$ to be its endpoint. In particular, this defines $\rho_v\subset\hat f(\widehat{F_{K(v)}})$ and $\widetilde f(v)$. For each $w\in K_1$, we also let $\tau_w\subset O$ be an arbitrary singular geodesic whose ultralimit is $\rho_w$.

We must show the cliques to orthants property and injectivity of $\widetilde f$.

\medskip\noindent
\emph{Cliques to orthants:}  
Let $C \subset \Gamma$ be a clique composed of vertices of complementary degree at most $k$.
Since $\rho_u$ is singular for all $u\in C$, the property follows for all subcliques of $C$, so we may assume that $C$ is maximal.
If $C$ does not contain vertices of complementary degree $k$, then the conclusion follows from the induction hypothesis. Otherwise, let $v \in C$ be a vertex of complementary degree $k$. If $C=K(v)$ then by construction we know that the rays $\{\rho_u\}_{u\in K(v)}$ span an orthant of $\hat f(\widehat F_{K(v)})$. Let $ B \subset C$ denote the branch-complemented vertices of $C$. 

\begin{claim*}
    $C\subseteq B\cup K(v)$
\end{claim*}

\begin{claim*proof}
If $u\in C\ssm B$, then neither $v$ nor $u$ is branch-complemented. By \cref{it:edge_Kcomp}, either $u \in K(v)$ and we are done, or $v \in K(u)$. In this case, $K(v) \subseteq K(u)$, and hence $k=|K(v)|\leq |K(u)|\leq k$, so the sets are equal and $u \in K(u) = K(v)$.
\end{claim*proof}

\begin{claim*}
    $B\cup K(v)$ is a clique.
\end{claim*}

\begin{claim*proof}
Let $b \in B \ssm K(v)$. Since both $b$ and $v$ belong to $C$, they are adjacent. 
The edge $\{v,b\}$ does not satisfy \cref{it:edge_branch_complemented} because $v$ is not branch-complemented. It does not satisfy \cref{it:edge_Kcomp}, as $b \notin K(v)$ and $K(b)=\{b\}$.
Thus \cref{it:edge_complementary} is satisfied and so $v$ is contained in some complementary clique $K$ of $b$, and by definition $K(v) \subseteq K$. In particular, $b$ is adjacent to all vertices in $K(v)$. 
\end{claim*proof}

Because $C$ is maximal, we have $C=B\cup K(v)$. As observed above we are done if $C=K(v)$, so suppose that $B'=B\ssm K(v)$ is nonempty. We have $C=B'\sqcup K(v)$, and hence the cliques to orthants property for $C$ follows by applying \Cref{prop:branch_complemented_splitting_cone}.

\medskip\noindent
\emph{Injectivity.}
Let $v$ and $v'$ be distinct vertices of complementary degree at most $k$. Set $K = K(v)$ and $K'=K(v')$. If $v' \in K $, then $\widetilde f(v) \ne \widetilde f(v')$ by construction, so assume that $v'\notin K$. 
By construction, $\widetilde f(v)\in\partial_T\hat f(\widehat{F_{K}})$ and $\widetilde f(v')\in\partial_T\hat f(\widehat{F_{K'}})$. If $K\cap K'=\varnothing$, then $\partial_T\hat f(\widehat{F_{K}})\cap \partial_T\hat f(\widehat{F_{K'}})=\varnothing$, so $\widetilde f(v)\ne\widetilde f(v')$.

Otherwise, $K\cap K'=\{u_1,\dots,u_p\}$ is nonempty. Since $v'\notin K$, we have $p<k$. Let $\rho_1, \dots, \rho_p \in \hat Y$ be the singular geodesic rays based at the cone point of $\hat Y$ such that $\partial_T \rho_i = \widetilde f(u_i)$. Let $R_v \subset \hat Y$ (respectively $R_{v'}$) be the singular orthant based at the cone point that is spanned by the images of vertices of $K(v)$ (respectively $K(v')$). By construction, $R_v \subset \hat f(\widehat{F_{K}})$ and $R_{v'} \subset \hat f(\widehat{F_{K'}})$, so 
\[
 \rho_1,\dots,\rho_p\subset R_v \cap R_{v'} \,\subset\, \hat f(\widehat{F_{K}}) \cap \hat f(\widehat{F_{K'}}) \,=\, \hat f(\widehat{F_{K\cap K'}}),
\]
where the equality holds because $\hat f$ is a biLipschitz embedding, and in particular injective. If $v$ and $v'$ were assigned the same singular geodesic ray $\rho$ based at the cone point of $\hat Y$, then we would have $\rho\subset R_v\cap R_{v'}$. But $\rho$ must be distinct from $\rho_1,\dots,\rho_p$, because $\{v,u_1,\dots,u_p\}\subset K$, so $\{\rho,\rho_1,\dots,\rho_p\}$ together span a singular $(p+1)$--orthant inside $R_v\cap R_{v'}\subset\hat f(\widehat{F_{K\cap K'}})$, which is impossible. This shows that $\widetilde f$ is injective, concluding the induction. As described above, this completes the proof.
\end{proof}

\begin{remark} \label{rem:bg_asrk}
If we make the weaker assumption in \cref{thm:beyond_FB_induced_embedding_singboundary} that $Y$ is a finite-dimensional CAT(0) cube complex with \emph{asymptotic rank} $n$, then with only small modifications to the above proof we can still obtain a combinatorial embedding $\tilde f:\Gamma\to\partial_{\mathrm{sing}}\hat Y$. 

Indeed, the map $\tilde f$ will be the one constructed in the proof of \cref{thm:beyond_FB_induced_embedding_singboundary}, and there is no need to define $\tau_v$. Instead of applying \cite[Thm~1.1]{huang:top} and \cite[Prop.~5.7]{baderbensaidpetyt:from}, we can apply \cite[Thm.~6.13]{baderbensaidpetyt:from}, which states: if $h:X\to Z$ is a quasiisometric embedding of finite-dimensional CAT(0) cube complexes of asymptotic rank $n$ and if $H\subset X$ is a flat that is an intersection of $n$--flats, then $h(H)$ lies at finite Hausdorff distance from a finite union of \emph{semisingular} orthants. In our setting, this gives singular orthants in $\hat Y$, which is all that is needed for the above proof. The remainder of the induction step is unchanged. 
\end{remark}

In the same spirit as \cref{rem:bg_asrk}, we can replace the codomain in \cref{thm:beyond_FB_induced_embedding_singboundary} by a product of $n$ hyperbolic spaces.

\begin{theorem} \label{rem:bg_hyp}
Let $\Lambda$ be a graph with clique number $n$, and let $\Gamma$ be a branch-graded subgraph. If there exists a quasiisometric embedding $f:A_\Lambda\to Y$ where $Y$ is a product of $n$ hyperbolic spaces, then $\Gamma$ is $n$--colourable.
\end{theorem}

As with \cref{rem:bg_asrk}, the proof of \cref{rem:bg_hyp} is very similar to that of \cref{thm:beyond_FB_induced_embedding_singboundary}, but this time there is the added subtlety that $Y$ is not itself a CAT(0) cube complex. This can be handled with a similar sleight of hand to the proof of \cite[Cor.~11.6]{baderbensaidpetyt:from} that does not affect the main idea, so we just provide an outline of the necessary changes to the proof of \cref{thm:beyond_FB_induced_embedding_singboundary}.

\begin{proof}[Sketch proof of \cref{rem:bg_hyp}]
Asymptotic cones of hyperbolic spaces are real trees, so $\partial_{\mathrm{sing}}\hat Y$ is a join of $n$ discrete graphs. Moreover, there exist trees $T_1,\dots, T_n$ such that the product $\hat Y$ admits a factor-wise isometric embedding $\iota$ into $\hat Z$, where $Z=\prod_{i=1}^nT_i$. 
To show that $\Gamma$ is $n$--colourable, it therefore suffices to define a combinatorial embedding $\psi:\Gamma\to\partial_{\mathrm{sing}}\hat Z$. 
Since $Z$ is an $n$--dimensional CAT(0) cube complex, we regain access to results from \cite{baderbensaidpetyt:from}, and we follow the construction of $\tilde f$ in the proof of \cref{thm:beyond_FB_induced_embedding_singboundary}.

In the ``complementary degree 0'' case, instead of applying \cite[Thm~1.1]{huang:top} and \cite[Prop.~5.7]{baderbensaidpetyt:from} to a quasiisometric embedding into $Y$, we apply \cite[Prop.~9.20]{baderbensaidpetyt:from} to the composition of the induced bilipschitz embedding $\hat f$ of asymptotic cones with $\iota:\hat Y\to\hat Z$. This shows that $\iota\hat f(\gamma_v)\subset\hat Z$ is a singular geodesic, so we can define $\psi(v)$ to be an endpoint of this geodesic. For the ``cliques to orthants'' property we also use \cite[Prop.~9.20]{baderbensaidpetyt:from} on $\widehat{F'}$ instead of \cite[Cor.~9.21]{baderbensaidpetyt:from} on $F'$ to see that $\iota\hat f(\widehat{F_C})$ is a singular flat. 

In the induction step we again work only with asymptotic cones. Instead of using \cite[Thm~1.1]{huang:top} and \cite[Props.~5.7, 5.3]{baderbensaidpetyt:from}, we apply \cite[Lem.~5.2]{bowditch:quasiflats} to $\iota\hat f(\widehat{F_{K(v)}})$, which tells us that it is equal to a finite union of singular orthants in the product of trees $\hat Z$, and then \cite[Prop.~5.2]{baderbensaidpetyt:from}, which gives that the boundary of $\iota\hat f(\widehat{F_{K(v)}})$ is homeomorphic to $\mathbb S^{k-1}$. We then use those orthants to define $\psi$ on $K(v)$ in the same way that $\tilde f$ is defined in the proof of \cref{thm:beyond_FB_induced_embedding_singboundary}. The remainder of the inductive step is unchanged, except that \cref{prop:branch_complemented_splitting_cone} is applied to $\iota\hat f$ instead of $\hat f$.
\end{proof}

\Cref{rem:bg_hyp} for instance shows that the RAAGs corresponding to the first four graphs in \Cref{fig:exples_branch_graded} cannot be quasiisometrically embedded into any product of three hyperbolic spaces, because those graphs are not 3--colourable.

We conclude this section by providing a large, natural family of branch-graded graphs with clique number three.

\begin{definition}[Boundary-sparse triangulation]
Let $\Sigma=\Sigma_{g,b}$ be a compact surface with genus $g\ge0$ and with $b\ge0$ boundary components. Let $\Lambda$ be a finite simplicial triangulation of $\Sigma$. For a boundary component $B\subset\partial\Sigma$, let $\deg_\Lambda(B)$ denote the tuple obtained by reading cyclically along $B$ the degrees of the vertices of $\Lambda$ that are in $B$. We say that $\Lambda$ is \emph{boundary-sparse} if there is some $B\subset\partial\Sigma$ such that a cyclic permutation of $\deg_\Lambda(B)$ contains either a two or two consecutive threes.
\end{definition}

We call a vertex or edge of a simplicial triangulation of a surface $\Sigma$ \emph{internal} if it is not contained in $\partial\Sigma$.

\begin{observation}\label{obs:opposite_internal}
    Let $\Lambda$ be a finite simplicial triangulation of a compact surface. If $\Lambda$ is not boundary-sparse then every vertex is opposite in some triangle to an internal edge, because otherwise $\Lambda$ would be a cone triangulation of a disc, or obtained from one by contracting an internal edge.
\end{observation}

\begin{lemma} \label{lem:boundary_sparse}
Let $\Lambda$ be a finite simplicial triangulation of a compact surface $\Sigma$. If $\Lambda$ is not boundary-sparse then it is branch-graded.
\end{lemma} 

\begin{proof}
\cref{second_item} of the definition is automatic.  

Every internal edge of $\Lambda$ is the intersection of two triangles, so it is branching, and by \cref{obs:opposite_internal}, every vertex of $\Lambda$ is opposite in some triangle to an internal edge. Because $\Lambda$ is simplicial, every internal vertex and every vertex of degree at least four is the intersection of two internal edges, so it is branching and hence branch-complemented.

Thus the only vertices that can fail to be branch-complemented are boundary vertices of degree three. To show \cref{first_item}, we will show that each such vertex lies in a branch-complemented edge. Let $v$ be one such vertex. Let $e_1=\{v,v_1\}$ and $e_2=\{v,v_2\}$ be edges of $\Lambda$ incident to $v$ with $e_1\subset\partial\Sigma$ and $e_2\not\subset\partial\Sigma$. We know that $e_2$ is branching. Since $\Lambda$ is a triangulation and $v$ has degree three, $\{v,v_1,v_2\}$ is a triangle in $\Lambda$. Since $\Lambda$ is not boundary-sparse, $v_1$ has degree at least four, so it is branching. Thus, $e_2$ is branch-complemented. Hence \cref{first_item} is satisfied by $\Lambda$, and $K_{\Lambda}(v)=\{v,v_2\}$, since boundary edges are not branch-complemented.


We can now check \cref{third_item}. Let $\{u,v\}$ be an edge of $\Lambda$. If $u$ and $v$ are branch-complemented, then \cref{it:edge_branch_complemented} holds. If $v$ is not branch-complemented, then it is a boundary vertex of degree three. If $\{u,v\}$ is a boundary edge, then the previous paragraph shows that \cref{it:edge_complementary} is satisfied by $\{u,v\}$. Otherwise the previous paragraph shows that $u\in K_\Lambda(v)$, so \cref{it:edge_Kcomp} is satisfied.
\end{proof}

Combining \cref{thm:beyond_FB_induced_embedding_singboundary,lem:boundary_sparse} we get the following.

\begin{corollary}
Let $\Lambda$ be a simplicial triangulation of a compact surface $\Sigma$ and let $Y$ be a 3--dimensional CAT(0) cube complex. If $\Lambda$ is not boundary-sparse and $A_\Lambda$ can be quasiisometrically embedded into $Y$, then $\Lambda$ is a subgraph of $\partial_{\mathrm{sing}}Y$.
\end{corollary}

\section{Quasiisometric embeddings into products of trees} \label{sec:embeddings_into_trees}

In this short section we use the results of \cref{sec:beyond_FB_2} to obstruct embeddings of RAAGs into products of trees and hyperbolic spaces, establishing \cref{mthm:rull}. We first consider the general branch-graded case (\cref{thm:fb_in_product_trees}), then refine it in the 2--dimensional setting (\cref{cor:fb_into_products_of_trees}). 

In the following, we allow trees to have infinite valence. We think of $(F_2)^n$ as a product of 4--regular trees. 

\begin{theorem} \label{thm:fb_in_product_trees}
Let $A_\Gamma$ be a branch-graded, $n$--dimensional RAAG. The following are equivalent.
\begin{enumerate}
\item   $\Gamma$ is $n$--colourable.\label{item:n_colourable}
\item   $A_\Gamma$ can be equivariantly isometrically embedded in a product of $n$ trees. \label{item:ie_in_trees}
\item   $A_\Gamma$ can be quasiisometrically embedded in $(F_2)^n$.\label{item:qie_in_F_2^n}
\item   $A_\Gamma$ can be quasiisometrically embedded in a product of $n$ hyperbolic spaces. \label{item:qie_in_trees}
\end{enumerate}
\end{theorem}

The following lemma shows that \ref{item:n_colourable} implies \ref{item:ie_in_trees}. The lemma is well known; see \cite{dranishnikovjanuszkiewicz:every,button:groups,holloway:embeddings}. We include a proof for completeness.

\begin{lemma} \label{lem:colourable_isometric_embedding}
If $\Gamma$ is an $n$--colourable graph, then $A_\Gamma$ can be equivariantly isometrically embedded in a product of $n$ trees.
\end{lemma}

\begin{proof}
Recall that $A_\Gamma$ is the 0--skeleton of the CAT(0) cube complex $X_\Gamma$. Given an $n$--colouring of $\Gamma$ with colours $c_1,\dots,c_n$, let $H_i$ be the set of hyperplanes of $X_\Gamma$ dual to edges labelled by generators corresponding to vertices of colour $i$. Note that no two elements of $H_i$ cross, and each $H_i$ is $A_\Gamma$-invariant. 

Let $T_i$ denote the graph with: a vertex $v_C$ for each connected component $C$ of the complement $X_\Gamma\ssm\bigcup_{h\in H_i}h$; and an edge from $v_C$ to $v_{C'}$ whenever $C$ and $C'$ are separated by a single hyperplane. Since no two elements of $H$ cross, the graph $T_i$ is a tree. (This is a special case of Sageev's construction.) 

There is a natural 1--Lipschitz map $f_i:A_\Gamma\to T_i$ given by sending $x\in A_\Gamma\subset X_\Gamma$ to the vertex $v_C$ of $T_i$ such that $x\in C$. This map is $A_\Gamma$-equivariant, because $H_i$ is $A_\Gamma$-invariant. Assembling the $f_i$ together gives an equivariant map $f:A_\Gamma\to\prod_{i=1}^nT_i$. The distance between two elements of $A_\Gamma$ is equal to the number of hyperplanes that separate them, which is easily seen to be the same as the distance between their images under $f$. Thus $f$ is an isometric embedding.
\end{proof}

The trees constructed in the proof of \cref{lem:colourable_isometric_embedding} are almost always locally infinite.

\begin{proof}[Proof of \cref{thm:fb_in_product_trees}]
If $\Gamma$ is $n$--colourable, then \cref{lem:colourable_isometric_embedding} shows that $A_\Gamma$ isometrically embeds in a product of $n$ trees, and Rull's theorem \cite[Thm~1.1]{rull:embedding} shows that $A_\Gamma$ quasiisometrically embeds in $F_2^n$. Thus \ref{item:n_colourable} implies \ref{item:ie_in_trees} and \ref{item:qie_in_F_2^n}.
It is obvious that \ref{item:ie_in_trees} and \ref{item:qie_in_F_2^n}, both imply \ref{item:qie_in_trees}. The fact that \ref{item:qie_in_trees} implies \ref{item:n_colourable} is given by \cref{rem:bg_hyp}.
%
\end{proof}

Note that only the implication \ref{item:qie_in_trees} $\Rightarrow$ \ref{item:n_colourable} used the branch-graded assumption.


A graph is 2-colourable if and only if it does not contain an induced odd cycle. Since cycles of length at least four are branch-graded (see \cite[\S3]{baderbensaidpetyt:quasiisometric:rigidity}), the same argument as in \cref{thm:fb_in_product_trees} gives the following. 

\begin{corollary} \label{cor:fb_into_products_of_trees}
Let $A_\Gamma$ be a right-angled Artin group. The following are equivalent. 
\begin{enumerate}
\item   $\Gamma$ is $2$--colourable.
\item   $A_\Gamma$ can be equivariantly isometrically embedded in a product of two trees. 
\item   $A_\Gamma$ can be quasiisometrically embedded in $F_2\times F_2$.
\item   $A_\Gamma$ can be quasiisometrically embedded in a product of two hyperbolic spaces. 
\end{enumerate}
\end{corollary}

The following can be viewed as a geometric version of \cite[Cor.~8.1(2)]{kimkoberda:embedability}, which states that if $\Gamma$ and $\Lambda$ are finite graphs such that $\Lambda$ is bipartite and $A_\Gamma<A_\Lambda$, then $\Gamma$ is bipartite.

\begin{corollary}
Let $\Lambda$ be a bipartite graph. If a right-angled Artin group $A_\Gamma$ quasiisometrically embeds in $A_\Lambda$, then $\Gamma$ is bipartite. 
\end{corollary}

\begin{proof}
\cref{cor:fb_into_products_of_trees} states that $A_\Lambda$ quasiisometrically embeds in a product of two trees, so $A_\Gamma$ must do also, and then the same corollary tells us that $\Gamma$ is bipartite.
\end{proof}



We are still left with some unanswered questions.

\begin{question}
Is the assumption that $\Gamma$ is branch-graded necessary in \cref{thm:fb_in_product_trees}? 
\end{question}

In view of \cref{lem:colourable_isometric_embedding} and \cite{rull:embedding}, this is essentially equivalent to asking whether every $n$--dimensional RAAG that quasiisometrically embeds in a product of $n$ trees has $n$--colourable defining graph. For example, does $A_{C_5} \times \mathbb Z$, whose defining graph is the cone over $C_5$ and is not $3$-colourable, quasiisometrically embed into a product of three trees? The closest we can get to answering this question is by adding three external edges (see the first two examples in \Cref{fig:exples_branch_graded}). This makes the cone a branch-graded subgraph, and hence the corresponding RAAGs do not quasiisometrically embed into a product of three trees by \cref{rem:bg_hyp}. 

A second aspect that is of interest is the requirement that the domain and codomain have the same dimension. In general, quasiisometric embeddings between objects of different dimensions are hard to understand, but perhaps the following can be answered.

\begin{question}
Suppose $\Gamma$ is triangle-free but not 2--colourable. Can $A_\Gamma$ be quasiisometrically embedded in a product of three trees?
\end{question}

In particular, we do not know how many trees are needed to quasiisometrically embed the RAAG on a triangle-free graph of chromatic number four.

\section{Isometric embeddings into products of trees}\label{sec:isometric_embeddings}

In view of \cref{thm:fb_in_product_trees}, it is natural to ask about isometrically embedding RAAGs in finite products of finite-valence trees. In this section, we show that this is only possible if the RAAG is itself a product of free groups; see \cref{cor:ie_prod_free}.

\begin{proposition} \label{prop:no_ie_in_finite_valence}
The group $\mathbb{Z}^2*\mathbb{Z}=\langle a,b,c\mid aba^{-1}b^{-1}\rangle$ with its standard word metric does not admit an isometric embedding in a finite product of locally finite trees with the $\ell^1$--metric.
\end{proposition}

\begin{proof}
Let $A=\sgen{a,b,c\mid aba^{-1}b^{-1}}\cong \mathbb{Z}^2*\mathbb{Z}$.
Suppose that $f:A\to\prod_{i\in I}T_i$ is an isometric embedding, where $T_i$ is a locally finite tree for each $i\in I$. We will show that $I$ is infinite. Let $Y=\prod_{i\in I}T_i$. The distance between two elements of $A$ is equal to the number of hyperplanes of $X_\Gamma$ that separate them, and their $f$--images must be separated by the same number of hyperplanes in $Y$. In particular, if $\gamma\subset A$ is a geodesic, then $f(\gamma)$ cannot cross any hyperplane of $Y$ twice.

Both $f(\sgen a)$ and $f(\sgen b)$ are geodesics. For an edge $e\subset A$, let $h_e$ be the hyperplane of $Y$ dual to the edge $f(e)$. Let $T_{i_1}$ be the tree corresponding to the edge $(f(1),f(a))$, that is, the unique tree into which the edge projects injectively. Since $a$ and $b$ commute, the hyperplane $h_{(1,a)}$ crosses $h_{(b^k,b^{k+1})}$ for all $k$. In particular, none of the $h_{(b^k,b^{k+1})}$ come from $T_{i_1}$. Because of this, the projection of $f(b^k)$ to $T_{i_1}$ is equal to the projection of $f(1)$ for all $k$.

Consider the edges $(f(b^k),f(b^kc))$ of $Y$, with $k\in\mathbb N$. Each one has a vertex whose projection to $T_{i_1}$ is equal to the projection of $f(1)$. Since $T_{i_1}$ is locally finite, only finitely many of these edges can come from $T_{i_1}$. Thus there is some $j_1\ne i_1$ such that infinitely many of them come from $T_{j_1}$. Let $K_1=\{k\in\mathbb N\,:\,(f(b^k),f(b^kc))\text{ comes from }T_{j_1}\}$. Note that if $k\in K_1$, then the projection to $T_{i_1}$ of $f(b^kc)$ is equal to that of $f(b^k)$ and hence to that of $f(1)$.

For $k\in K_1$, consider the geodesic rays $f(b^kc\sgen a^+)$ and $f(b^kc\sgen b^+)$. We claim that only finitely many such rays can have edges that come from $T_{i_1}$. Indeed, if there were infinitely many then there would be two such rays $\gamma_1$ and $\gamma_2$ such that the first edge of $\gamma_1$ coming from $T_{i_1}$ comes from the same edge of $T_{i_1}$ as the first edge of $\gamma_2$. But if that happened then the concatenation of $\gamma_1$ and $\gamma_2$, which is a geodesic in $A$ would satisfy that its $f$--image crosses some hyperplane of $Y$ twice, contradicting the assumption that $f$ is an isometric embedding.

We have shown that there is some $k_1\in K_1$ such that no geodesic ray $f(b^{k_1}c\sgen a^+)$ or $f(b^{k_1}c\sgen b^+)$ has an edge coming from $T_{i_1}$. Let $T_{i_2}$ be the tree corresponding to the edge $(f(b^{k_1}c),f(b^{k_1}ca))$. Note that $i_2\ne i_1$. In particular, $|I|\ge2$ 

We now repeat the above argument. Namely, $h_{(b^{k_1}c,b^{k_1}ca)}$ crosses $h_{(b^{k_1}cb^k,b^{k_1}cb^{k+1})}$ for all $k$, and by local finiteness of $T_{i_1}$ and $T_{i_2}$ there is some $j_2\notin\{i_1,i_2\}$ such that infinitely many of the edges $(f(b^{k_1}cb^k),f(b^{k_1}cb^kc))$ come from $T_{j_2}$. Letting $K_2$ be the set of such $k$, we can argue as above, again using local finiteness of $T_{i_1}$ and $T_{i_2}$, that there is some $k_2\in K_2$ such that no geodesic ray $f(b^{k_1}cb^{k_2}c\sgen a^+)$ or $f(b^{k_1}cb^{k_2}c\sgen b^+)$ has an edge coming from either $T_{i_1}$ or $T_{i_2}$. Let $T_{i_3}$ be the tree corresponding to the edge $(f(b^{k_1}cb^{k_2}c),f(b^{k_1}cb^{k_2}ca))$. Noting that $i_3\notin\{i_1,i_2\}$, we must have $|I|\ge3$.

By iterating this process, we see that the set $I$ indexing the product of trees $Y=\prod_{i\in I}T_i$ cannot be finite.
\end{proof}

The proof of the following lemma is purely graph-theoretic. A \emph{parabolic} subgroup of a RAAG $A_\Gamma$ is a subgroup generated by a subset of the standard generators of $A_\Gamma$.

\begin{lemma} \label{lem:tree_flats_parabolic}
Let $A_\Gamma$ be a right-angled Artin group. Either $A_\Gamma$ is a direct sum of free groups, or $\Z^2*\Z$ is a parabolic subgroup of $A_\Gamma$.
\end{lemma}

\begin{proof}
Assume $A_\Gamma$ does not contain a parabolic $\Z^2*\Z$ subgroup. 
Equivalently, $\Gamma$ does not contain the disjoint union of an edge and a vertex as an induced subgraph. We will refer to this as the \emph{graph assumption}. We need to show that $\Gamma$ decomposes as a join of edgeless graphs.

Consider the following relation $\sim$ on the vertices of $\Gamma$:
$$
u \sim v \iff u=v \text{ or } u \text{ is not adjacent to } v.
$$
The relation $\sim$ is reflexive and symmetric by definition.
It is also transitive. Indeed, assume that $u \sim v$ and $v \sim w$. If $u \not\sim w$, then $\{u,v\}$ forms an edge of $\Gamma$ and the induced subgraph on $\{u,v,w\}$ is a disjoint union of that edge and $\{w\}$, in contradiction to the graph assumption. Thus $\sim$ is an equivalence relation.

Let $\{C_i\,:\,i\in I\}$ be the equivalence classes of $\sim$. By definition, each $C_i$ has no edges. Moreover, if $i\neq j\in I$, then every vertex of $C_i$ is adjacent to every vertex of $C_j$, otherwise two non-adjacent vertices in distinct classes would be equivalent. Therefore $\Gamma=\bigast_{i\in I}C_i$,
and hence $A_\Gamma=\bigoplus_{i\in I}A_{C_i}$ is a direct sum of free groups.
\end{proof}

\begin{corollary} \label{cor:ie_prod_free}
A right-angled Artin group can be isometrically embedded in a finite product of locally finite trees if and only if it is a finite product of finite-rank free groups.
\end{corollary}

\begin{proof}
The reverse direction of the equivalence is obvious, so let $A_\Gamma$ be a right-angled Artin group that isometrically embeds in a finite product of locally finite trees. The graph $\Gamma$ must be finite, for otherwise $A_\Gamma$ would be locally infinite. By \cref{prop:no_ie_in_finite_valence}, $\Z^2*\Z$ is not a parabolic subgroup of $A_\Gamma$. By \cref{lem:tree_flats_parabolic}, we deduce that $A_\Gamma$ is a finite direct product of finite-rank free groups.
\end{proof}

\section{Stable embeddings into products} \label{antirull}

The goal of this section is to provide a counterpoint to Rull's theorem \cite{rull:embedding}. As noted in the introduction, Rull's quasiisometric embedding of an $n$--colourable RAAG into $(F_2)^n$ is almost never stable. We shall prove the necessity of this by considering stable quasiisometric embeddings of RAAGs into products of RAAGs. In particular, we show in \cref{cor:no_stable_Rull} that if $A_\Gamma$ can be stably quasiisometrically embedded in a finite product of finite-rank free groups, then $A_\Gamma$ is itself a finite product of free groups.

The following appears as \cite[Def.~6.1]{baderbensaidpetyt:quasiisometric:rigidity}.

\begin{definition}[Stable] \label{def:stable_QIE}
A quasiisometric embedding $\varphi:A_\Gamma\to A_\Lambda$ is said to be \emph{stable} if there is a constant $D$ such that $\varphi$ sends every standard geodesic of $A_\Gamma$ within Hausdorff distance at most $D$ of some standard geodesic of $A_\Lambda$.
\end{definition}

We shall use the notion of \emph{extension graphs}, introduced in \cite{kimkoberda:embedability}. Extension graphs are sometimes thought of as playing the role of the curve graph of a mapping class group, and they are used heavily in both \cite{baderbensaidpetyt:quasiisometric:flexibility} and \cite{baderbensaidpetyt:quasiisometric:rigidity}.

\begin{definition}[Extension graph] \label{def:extension_graph}
For a graph $\Gamma$, the \emph{extension graph} $\Gamma^\ext$ is the following graph. It has a vertex for each parallelism class of standard geodesics in $X_\Gamma$. Two parallelism classes are joined by an edge if they admit representatives that span a 2--flat of $X_\Gamma$. 

For $v\in \Gamma$ with corresponding standard generator $s$, we denote the vertex of $\Gamma^\ext$ corresponding to $g\sgen s$ by $g\cdot v$. 
\end{definition}

The following connects stable quasiisometric embeddings with extension graphs; see \cite[Prop.~6.2]{baderbensaidpetyt:quasiisometric:rigidity}.

\begin{proposition} \label{prop:stable_implies_induced}
Let $\Gamma$ and $\Lambda$ be graphs and let $\varphi:A_\Gamma\to A_\Lambda$ be a stable quasiisometric embedding. The map $\varphi$ induces a combinatorial embedding $\Gamma^\ext\to \Lambda^\ext$. 
\end{proposition}

We aim to control the images of standard 2--flats under stable quasiisometric embeddings, and we can do so for standard flats that avoid the following degeneracy condition.

\begin{definition}[Spinal]
We call an edge $e$ of a graph $\Gamma$ \emph{spinal} if every vertex of $\Gamma$ is adjacent to $e$. We refer to a 2--flat $F\subseteq A_\Gamma$ as being \emph{spinal} if it is a standard flat whose corresponding generators are the vertices of a spinal edge of $\Gamma$. 
\end{definition}

If $\gamma=g\sgen s\subset X_\Gamma$ is a standard geodesic, then we write $\type(\gamma)=s$ and refer to $s$ as the \emph{type} of $\gamma$.

\begin{theorem} \label{thm:anti-rull2}
Let $\Gamma$ be a graph, and let $\Lambda$ be a finite graph that decomposes as a join $\Lambda=\bigast_{i=1}^r\Lambda_i$. Suppose there exists a stable quasiisometric embedding $\varphi:A_\Gamma\to A_\Lambda$. Each non-spinal standard 2--flat of $A_\Gamma$ is mapped within a uniformly finite Hausdorff distance of a 2--flat coming from a single $\Lambda_i$.
\end{theorem}

\begin{proof}
Suppose that $\varphi:A_\Gamma\to A_\Lambda$ is a stable quasiisometric embedding, so that $\varphi$ induces a map $f:\Gamma^\ext\to\Lambda^\ext$, by \cref{prop:stable_implies_induced}. Since $\varphi$ is stable, by \cite[Prop.~4.3]{baderbensaidpetyt:from} $\varphi$ maps each standard 2--flat of $A_\Gamma$ within a uniformly bounded Hausdorff distance of some singular 2--flat of $A_\Lambda$. By \cite[Lem.~4.2]{baderbensaidpetyt:from} and \cite[Lem.~2.14 (3)]{baderbensaidpetyt:quasiisometric:rigidity}, that 2--flat is standard. 

Suppose that there is some non-spinal 2--flat $F\subseteq A_\Gamma$ whose corresponding image 2--flat $F'$ does not arise from a single $\Lambda_i$. Our goal is to derive a contradiction by finding a pair of non-parallel geodesics in $A_\Gamma$ with parallel $\varphi$--images.

By post- and pre-composing $\varphi$ with left translations, we can assume that $1\in F'$ and $F$ is a standard flat corresponding to a non-spinal edge $e\subseteq\Gamma\subseteq\Gamma^\ext$. After relabelling the $\Lambda_i$, we can assume that the standard generators of $A_\Lambda$ that generate $F'$ are $t_1\in A_{\Lambda_1}$ and $t_2\in A_{\Lambda_2}$. 

Let $v\in\Gamma$ be a vertex that is not a neighbour of $e$, with corresponding generator $s\in A_\Gamma$. If $g_1,g_2\in F$ are distinct, then the standard geodesics $g_1\sgen s$ and $g_2\sgen s$ are disjoint and diverge from each other and $F$ at constant speed; in particular, $f(g_1\cdot v)\neq f(g_2\cdot v)$. If $g\in F$, then since $g\sgen s$ meets $F$, we have that $\varphi(g\sgen s)$ comes $D$--close to $F'$. Since $\varphi$ is stable, we can fix a standard geodesic $\gamma_g\subseteq A_\Lambda$ that is within uniform Hausdorff distance of $\varphi(g\sgen s)$, comes within a uniform distance of $F'$, and represents the vertex $f(g\cdot v)$ of $\Lambda^\ext$. 

Say that a point $x\in F'$ is \emph{hit} if there is an element $g_x\in F$ such that $x$ is the $\ell^1$--closest point in $F'$ to $\gamma_{g_x}$. The element $g_x$ need not be unique, so we fix a choice of $g_x$ for each $x\in F'$ that is hit. Note, however, that each $g\in F$ is responsible for at most one $x\in F'$ being hit. If $x\in F'$ is hit, then let $w_x$ be the minimal element of $A_\Lambda$ such that $xw_x\in\gamma_{g_x}$. By construction, there is some $K$ such that $|w_x|\le K$ for every $x$ that is hit. 

Since $\varphi$ is a quasiisometric embedding that maps $F$ within finite Hausdorff distance of $F'$, there is some $R$ such that every ball in $F'$ of radius $R$ contains a point that is hit. For $m,n\in\Z$, let $Q_{m,n}$ be the rectilinear square in $F'=\sgen{t_1,t_2}$ that has side-length $2R$ and is centred on $t_1^{10Rm}t_2^{10Rn}$. Note that the $Q_{m,n}$ are pairwise disjoint, and each contains at least one point that is hit.

Each $Q_{m,n}$ contains $4R^2$ points. Since $\Lambda$ is finite, there are finitely many words in $A_\Lambda$ of length at most $K$ and finitely many types of standard geodesics. Thus there is a number $M$ such that any collection of $M$ of the squares $Q_{m,n}$ must contain a pair with ``the same relative point being hit in the same way''. More precisely, if $(m_1,n_1),\dots,(m_M,n_M)$ are distinct, then there exists: an $x\in Q_{0,0}$; an element $w\in A_\Lambda$ with $|w|\le K$; some standard generator $t\in A_\Lambda$; and a pair $(m_i,n_i)$, $(m_j,n_j)$ with the following properties. Both $x_i=xt_1^{10Rm_i}t_2^{10Rn_i}\in Q_{m_i,n_i}$ and $x_j=xt_1^{10Rm_j}t_2^{10Rn_j}\in Q_{m_j,n_j}$ are hit; we have $w_{x_i}=w_{x_j}=w$; and we have $\type(\gamma_{g_{x_i}})=\type(\gamma_{g_{x_j}})=t$. We encode the data of this pair by the tuple $\tau=(x,w,t,i,j)$.

For fixed $n$, consider the set $\cQ_n$ of squares $\{Q_{1,n},\dots,Q_{M,n}\}$. By the previous paragraph, there is a pair among them as above, encoded by a tuple $\tau_n$. Since there are only finitely many possible tuples, there is some $N$ such that the set $\{\cQ_1,\dots,\cQ_N\}$ contains two elements $\cQ_i$ and $\cQ_j$ with the same associated tuple $\tau_i=\tau_j$.

From this, we have ``a rectangle of points being hit in the same way''. More precisely, we have shown that there exists: a point $x\in Q_{0,0}$; an element $w\in A_\Lambda$ with $|w|\le K$; a standard generator $t\in A_\Lambda$; and numbers $m_1,m_2\in\{1,\dots,M\}$ and $n_1,n_2\in\{1,\dots,N\}$ such that each $x_{ij}=xt_1^{10Rm_i}t_2^{10Rn_j}$ with $i,j\in\{1,2\}$ is hit, and also $w_{x_{ij}}=w$ and $\type(\gamma_{g_{x_{i,j}}})=t$. 

Since $\Lambda=\bigast_{k=1}^r\Lambda_k$, because we found a rectangle of points hit in the same way, we can relabel so that the standard generator $t$ arises from some $\Lambda_k$ with $k\ne 1$. This decomposition moreover allows us to write $w=w_1w_2\dots w_r$, where $w_i\in A_{\Lambda_i}$. In particular, each $w_k$ commutes with each $w_{k'}$ with $k'\ne k$, and each standard generator $t_k$ coming from $\Lambda_k$ commutes with all $w_{k'}$ with $k'\ne k$. 

Since each $g\in F$ is responsible for at most one $x\in F'$ being hit, $g_{x_{11}}$ and $g_{x_{21}}$ are distinct. By construction, we have $\gamma_{g_{x_{ij}}}=x_{ij}w\sgen t$. From the above commutation relations, we see that each of $\gamma_{g_{x_{11}}}$ and $\gamma_{g_{x_{21}}}$ is parallel to $xt_2^{10Rn_1}w_2w_3\dots w_r\sgen t$. Thus $f(g_{x_{11}}\cdot v)=f(g_{x_{21}}\cdot v)$. But this is impossible, because we saw earlier that if $g_1,g_2\in F$ are distinct then $f(g_1\cdot v)\ne f(g_2\cdot v)$.
\end{proof}

\begin{remark} 
The proof of \cref{thm:anti-rull2} only uses a finite region of the 2--flat $F'$ of size depending on the various parameters, so one could get away with assuming something less strong but more technical than stability. 
\end{remark}

\begin{corollary} \label{cor:stable_irreducible}
Let $\Gamma$ be a connected graph with $\diam\Gamma>3$, and let $\Lambda$ be a finite graph that decomposes as a join $\Lambda=\bigast_{i=1}^r\Lambda_i$ of irreducible graphs. If $\varphi:A_\Gamma\to A_\Lambda$ is a stable quasiisometric embedding, then $\varphi$ induces a combinatorial embedding $\Gamma^\ext\to\Lambda_i^\ext$ for some~$i$.
\end{corollary}

\begin{proof}
Since $\varphi$ is stable, it induces an embedding $\Gamma^\ext\to\Lambda^\ext$, by \cref{prop:stable_implies_induced}. If $\diam\Gamma>3$, then no edge of $\Gamma$ is spinal. Hence there is some $i$ such that every standard 2--flat of $A_\Gamma$ is mapped within finite Hausdorff distance of a 2--flat coming from $\Lambda_i$, by \cref{thm:anti-rull2}. Thus the image of $\Gamma^\ext$ lies in $\Lambda_i^\ext$.
\end{proof}

\begin{remark} \label{rem:not_into_factor}
Note that one cannot conclude in \cref{cor:stable_irreducible} that $\varphi(A_\Gamma)$ lies in a uniform neighbourhood of a single irreducible factor. For example, given a non-complete graph $\Gamma$, let $\Lambda$ be the join of $\Gamma$ with another vertex, so that $A_\Lambda=A_\Gamma\times\Z$. Let $\varphi:A_\Gamma\to A_\Lambda$ be the map given by $\varphi(g)=(g,\syl g)$. It is a quasiisometric embedding because the map $g\mapsto\syl g$ is 1--Lipschitz, and it is stable because $\syl{\cdot}$ takes only two values on each standard geodesic. But $\varphi(A_\Gamma)$ is not contained in any finite neighbourhood of any coset of $A_\Gamma<A_\Lambda$.
\end{remark}

The following corollary echoes \cref{cor:ie_prod_free}. Recall that $\omega_\Gamma$ denotes the clique number of a graph $\Gamma$. 

\begin{corollary} \label{cor:no_stable_Rull}
If $A_\Gamma$ admits a stable quasiisometric embedding in a finite product of finite-rank free groups, then $A_\Gamma$ is itself a finite product of free groups.
\end{corollary}

\begin{proof}
Let $n=\omega_\Gamma$. If $n\le1$ then we are done, so suppose that $n\ge2$. We shall prove that $\Gamma$ is complete $n$--partite. Let $\varphi:A_\Gamma\to A_\Lambda$ be a stable quasiisometric embedding, where $\Lambda$ is a finite join of finite edgeless graphs $\Lambda_1,\dots,\Lambda_r$, and let $f:\Gamma^\ext\to\Lambda^\ext$ be the induced map of extension graphs.

Let $C=(v_1,\dots,v_n)$ be an $n$--clique of $\Gamma$. If $\Gamma=C$ then we are done. Otherwise, colour the vertex $v_i$ with colour $c_i$. Since no 2--flat of $A_\Lambda$ comes from any one $\Lambda_i$, \cref{thm:anti-rull2} tells us that every edge of $C$ is spinal. 

Let $v$ be any vertex of $\Gamma\ssm C$. It cannot be adjacent to every vertex of $C$, for then we would have $\omega_\Gamma>n$. But every edge of $C$ is spinal, so $v$ is adjacent to all but one vertex of $C$. If that vertex is $v_i$, then give $v$ colour $c_i$.

If $v,v'\in\Gamma\ssm C$ have colour $c_i$, then each is adjacent to all of $v_1,\dots,v_{i-1},v_{i+1},\dots,v_n$, so $v$ and $v'$ cannot be adjacent to each other because then we would have $\omega_\Gamma>n$. This shows that $\Gamma$ is $n$--partite. 

On the other hand, if $v,v'\in\Gamma\ssm C$ have distinct respective colours $c_i$ and $c_j$, then consider the edge $(v,v_j)$ of $\Gamma$. As it is spinal, one of its endpoints is adjacent to $v'$. By construction, it must be $v$. Thus $\Gamma$ is complete $n$--partite.
\end{proof}

\cref{cor:no_stable_Rull} is in stark contrast to Rull's theorem \cite[Thm~1.1]{rull:embedding}. This shows that methods like the Alice's diary construction used by Rull, or variants such as those used by Nairne \cite{nairne:embedding}, are necessary for constructing such embeddings.

More generally, but less strongly, we have the following, which is most interesting in the case $n=2$ where it constrains quasiisometric embeddings of 3--dimensional RAAGs into products of 2--dimensional RAAGs.

\begin{corollary} \label{cor:product_of_2dim}
Let $\Gamma$ be a connected graph with $\omega (\Gamma)\ge2n-1$, and let $\Lambda_1,\dots,\Lambda_r$ be finite graphs with $\omega_{\Lambda_i}\le n$ for all $i$. If there is a stable quasiisometric embedding $A_\Gamma\to A_\Lambda=\prod_{i=1}^rA_{\Lambda_i}$, then $\diam(\Gamma)\le2$.
\end{corollary}

\begin{proof}
If $n\le1$ then this follows from Corollary~\ref{cor:no_stable_Rull}. Otherwise, let $f:\Gamma^\ext\to\Lambda^\ext$ be the map induced by the stable quasiisometric embedding, as given by \cref{prop:stable_implies_induced}. Colour the vertices of $\Gamma\subseteq\Gamma^\ext$ according to which $\Lambda_i$ their image corresponds to. There can be no monochromatic $(n+1)$--clique in $\Gamma$. Moreover, \cref{thm:anti-rull2} implies that every edge between vertices of distinct colours is spinal.

Let $C_0\subseteq\Gamma$ be a $(2n-1)$--clique. If there exists a monochromatic $n$--subclique of $C_0$, then fix one such subclique $C$, and let $v\in C_0\ssm C$. Every edge from $v$ to $C$ is spinal. If instead $C_0$ has no monochromatic $n$--subcliques, then let $v\in C_0$ be arbitrary. Since at most $n-2$ other vertices of $C_0$ share a colour with $v$, there is an $n$--subclique $C\subset C_0$ such that no vertex of $C$ has the same colour as $v$. Again, every edge from $v$ to $C$ is spinal. 

Now let $w$ be an arbitrary vertex of $\Gamma$ that is not adjacent to $v$. Since every edge from $v$ to $C$ is spinal, $w$ is adjacent to every element of $C$. In particular, $\dist(w,v)=2$ and if $w'$ is any other vertex of $\Gamma$ not adjacent to $v$, then $\dist(w,w')\le2$. 

Now suppose that $v'$ is a vertex of $\Gamma$ that is adjacent to $v$. Any two such vertices are at distance at most two, so it only remains to show that $\dist(v',w)\le 2$. If $C$ is monochrome, then $w$ cannot have the same colour as it, for there are no monochromatic $(n+1)$--cliques in $\Gamma$, and hence every edge from $w$ to $C$ is spinal. Otherwise there is some vertex $u\in C$ whose colour is different to that of $w$, so the edge $(u,w)$ is spinal. In either case, $v'$ must be adjacent either to $w$ or to an element of $C$, which is adjacent to $w$.
\end{proof}


\bibliographystyle{alpha}
\footnotesize{\bibliography{biblio}}
\Addresses
\end{document}